\documentclass[11pt,oneside,a4papper]{osajnl}
\usepackage{amsthm}
\usepackage{comment}
\journal{jocn} 

\setboolean{shortarticle}{false}
\title{\centering 
Controllability of a system of non-autonomous degenerate coupled parabolic equations}

\author[1,3,*]{\centering \textcolor{black}{Alfredo S. Gamboa}}
\author[1]{\textcolor{black}{Juan Limaco }}
\author[1,2]{\textcolor{black}{Luis P. Yapu}}

\affil[1]{Universidade Federal Fluminense, Instituto de Matemática e Estatística, Niterói, Brazil}
\affil[2]{Friedrich-Alexander Universität Erlangen-Nürnberg (FAU), Chair of Dynamics, Control, Machine Learning and Numerics, Erlangen, Germany}
\affil[3]{Universidade Estadual do Rio de Janeiro, Escola Politécnica, Nova Friburgo, Brazil}
\affil[*]{\centering Contact: alfredo.soliz@iprj.uerj.br}

 \dates{}%\centering Compiled \today

\doi{}

\begin{abstract}
We prove a Carleman estimate for a one-dimensional parabolic equation which degenerates at one extremity of the domain and has a bounded, time dependent coefficient multiplying the diffusion term. Then we use the estimate to show the null controllability of a coupled system characterized by this form of diffusion operator and bounded coefficients.

\end{abstract}

\definecolor{mygreen}{RGB}{44,162,67}
\definecolor{mylilas}{RGB}{186,85,211}
\setboolean{displaycopyright}{false}

\newcommand{\cara}{\mathbb{1}}

\newtheorem{lema}{Lemma}
\newtheorem{teo}{Theorem}
\newtheorem{propo}{Proposition}
\newtheorem{coro}{Corollary}

\newcommand{\R}{\mathbb{R}}

\setboolean{displaycopyright}{false}

\begin{document}

\maketitle

\textbf{MSC Classification (2020)}: Primary: 35K65, 93B05; Secondary: 93C10. 

\textbf{keywords}: Degenerate parabolic equations, Non-autonomous systems, Controllability, Carleman inequalities.

\section{Introduction} \label{S:Intro}

%The problem we are considering is the following: 
Let $T>0$ be a fixed constant, $\Omega=(0,1)$ the spatial domain and $Q=\Omega \times (0,T)$ with lateral boundary $\Sigma = \partial \Omega \times (0,T)$.
We are interested in the following degenerate non-autonomous parabolic system of equations with Dirichlet boundary conditons,
\begin{equation}\label{eq:PDE}
	\left\{\begin{aligned}
		&u_t - b(t) \left(a(x) u_{x}\right)_{x} + d_1(x,t)\sqrt{a}u_x + b_{11}u + b_{12}v = h 1_{\omega} + H_1 &&\text{in} && Q,\\
        &v_t - b(t) \left(a(x) v_{x}\right)_{x} + d_2(x,t)\sqrt{a}v_x + b_{21}u + b_{22}v = H_2 &&\text{in}&& Q,\\
		&u(0,t)=u(1,t)=v(0,t)=v(1,t)=0&&\text{on} && (0,T), \\
		&u(\cdot,0) = u_0, \quad v(\cdot,0) = v_0 &&\text{in} && \Omega,
	\end{aligned}
	\right.
\end{equation}
where $u_{0}$, $v_0$ are the initial data, $h$ is the control, $\omega \subset \Omega$ is an open interval where the control acts
%, $\omega_T = \omega \times (0,T)$ 
and $1_A$ denotes the characteristic function of the set $A \times (0,T)$. The coefficients $d_i(x,t)$ and $b_{ij}(x,t)$, $i=1,2$, are bounded, and the time dependent function $b(t) \in W^{1,\infty}(0,T)$ is bounded from bellow by a constant $b_0 >0$. Moreover let us assume that there exists $\omega_1 \subset\subset \omega$ such that $\inf \{ b_{21}(x,t) \ : \ (x,t) \in \omega_1 \times [0,T] \}>0$. 
%$F$ is a $C^2$-function with bounded derivatives up to order 2.
%Let $\hat O_t$ and $\hat O_{i,t}$ denote open intervals included in $\hat \Omega_t$ and $\hat O = \cup_t \hat O_t$, $\hat O_{i} = \cup_t \hat O_{i,t}$ such that $\hat O_i \cap \hat O = \emptyset$.
%For $i=1,2$, we denote $\mathcal{V}_{i}=L^{2}(\hat O_{i})$ and 
%We use the notation $\mathcal{V}=L^{2}(\omega_T)$.

%We take degeneration functions $a(x')$ such that $a(\tau_t(x))$ has the form $h(t)a(x)$ with $x \in (0,1)$. In Section \ref{sec:diffeomorphism} the diffeomorphism $\tau_t$ has a specific form which verify this condition. 
%In this case, 
Moreover $a \in C([0,1]) \cap C^1((0,1])$ satisfies $a(0) = 0, a > 0 \text{ in } (0,1], a' \ge 0$ and  $xa'(x)\leq K a(x), \forall x \in [0,1] \text{ for some } K \in [0,1).$ In other words, the function $a$ behaves like  $x^\alpha, \text{ with } \alpha \in (0,1)$. This condition is called  \emph{weakly degenerate} in \cite{Alabau_cannarsa_fragnelli-06}.
Under these hypotheses, the function $\frac{x}{\sqrt{a}}$ is nondecreasing and thus is bounded from above by $\frac{1}{\sqrt{a(1)}}$.

For any $T>0$ we look for a control $h \in L^2(\omega \times (0,T))$ such that the system is \emph{null controllable} at time $T$, i.e. 
$$
u(x,T) = 0, \text{ for } x \in \Omega.
$$

As far as we know, the pioneering works on controllability studying the function spaces and the Carleman estimates for degenerate parabolic autonomous equations were \cite{Alabau_cannarsa_fragnelli-06} and \cite{Cannarsa_Martinez_Vancostenoble-2005}. The controllability of a degenerate coupled autonomous system has been studied in \cite{AitBenHassi-11, AitBenHassi-13}. 

Recently Akil, Fragnelli and Ismail \cite{fragnelli_no_autonomo} proved Carleman estimates for non-autonomous adjoint problems and proved null-controllability of the associated non-autonomous degenerate parabolic equations in divergent and non-divergent form. We observe that the weights in \cite{fragnelli_no_autonomo} are differents and that in our estimate (\ref{ecjv34_intro}) we do not have a factor $\frac{x^2}{a}$ in the second term of the left-hand side as is the case of the estimate in \cite{fragnelli_no_autonomo}. Moreover, our estimate (\ref{ecjv34_intro}) depends on two constants $s$ and $\lambda$ which can be chosen big enough independently as in the more standard Carleman estimates in the non-degenerate case, c.f. \cite{Fur_Ima-96, GlobalCarleman-06}, which gives more freedom for future applications.
The computations we follow are based on the work of Demarque, Límaco and Viana \cite{DemarqueLimacoViana_deg_eq2018} for an autonomous equation.
The case of Carleman estimate for a system of coupled degenerate autonomous parabolic equations with the weights we are using has been worked in the subsequent work  \cite{DemarqueLimacoViana_deg_sys2020}.

Some applications of degenerate parabolic equations appear in climate sciences \cite{FLORIDIA-2014, fragnelli_no_autonomo}. These papers treat the Budiko-Sellers model which models the surface temperature averaged over longitude and with degeneration function depending on the latitude on Earth and with degeneration at the poles. Other recent result related to the Budiko-Sellers model but with Robin boundary conditions is done in \cite{fragnelli_robin_BC-25} by the same authors of \cite{fragnelli_no_autonomo}. Another work in the degenerate case is related to the controllability of the Prandtl equation in boundary layer models \cite{Martinez_Raymond_Vancostenoble-2003}. Degenerate operators appear also in probability theory and Feller semigroups \cite{Feller-1952,Feller-1954}.

First we consider the case of a non-autonomous equation and we show a Carleman estimate for the linear (adjoint) equation
   \begin{equation}\label{eq:adjunto2_intro}
	\begin{cases}
		w_t+b(t)\left(a(x)w_x\right)_x+d(x,t)\sqrt{a}w_x + c(x,t)w  = h(x,t), & \ \ \ \text{en} \ \ \ {Q}\\
		w=0, & \ \ \ \text{in} \ \ \ {\Sigma}\\
        w(T) = w_T,& \ \ \ \text{in} \ \ \ (0,1)
	\end{cases}
\end{equation}
where $b$ belongs to $W^{1,\infty}(0,T)$. Moreover, we suppose that $b(t) \geq b_0 >0$ and $\frac{b'(t)}{b(t)} \leq B$ for some constants $b_0, B>0$. 

The main technical result of this paper is the following a Carleman estimate for the non-autonomous degenerate equation \eqref{eq:adjunto2_intro}, following the work of Demarque, Límaco and Viana for an autonomous degenerate equation \cite{DemarqueLimacoViana_deg_eq2018}. The weights and notations are explained in Section \ref{Sec:carleman_nonautonomous}.

\begin{teo}
\label{prop_carleman_1-intro}
	There exist $C>0$ and $\lambda_0,  s_0>0$ such that every solution $v$ of (\ref{eq:adjunto2_intro}) satisfies, for all
	$s\geq s_0$ and $\lambda\geq \lambda_0$,
	\begin{equation}\label{ecjv34_intro}
		\int_{0}^{T}\int_{0}^{1}e^{2s\varphi}\left(  (s\lambda)\sigma a b^2 w_x^2+(s\lambda)^2\sigma^2b^2 w^2 \right)\leq C\left(	\int_{0}^{T}\int_{0}^{1} e^{2s\varphi} |h|^2+(\lambda s)^3\int_{0}^{T}\int_{\omega}e^{2s\varphi} \sigma^3  w^2  \right).
	\end{equation}
\end{teo}

As an application, we show the global null controllability of our system \eqref{eq:PDE}. Since the equations are degenerate, the initial conditions $u_0$ and $v_0$ are taken in the weighted space $H_a^1(\hat \Omega)$; see \cite{Alabau_cannarsa_fragnelli-06} and Section \ref{Sec:carleman_nonautonomous}.

\begin{teo}
\label{thm:local_null_controllability}
Under the hypothesis considered in setting of \eqref{eq:PDE}, if moreover 
$\rho_2 H_1$, $\rho_2 H_2 \in L^2(Q)$, 
then, for any $T>0$ and any $u_0, v_0 \in H_a^1(\Omega)$, 
there exists a control $h \in L^2(\omega \times (0,T))$ such that the solution of \eqref{eq:PDE} is null-controllable at time $T$.
\end{teo}

%Our next result shows conditions under which the Nash quasi-equilibrium $(\hat v_1,\hat v_2)$ obtained in Theorem \ref{thm:local_null_controllability} is a Nash equilibrium.

%\begin{propo} 
%\label{thm:nash_equilibrium}
%Let $u_0 \in H_a^1(\hat \Omega_0)$.
%Under the hypotheses of Theorem \ref{thm:local_null_controllability}, 
%\textcolor{red}{there exists a constant $\delta_0 > 0$ such that if
%$$
%\|y^0\|_{H_a^1(\Omega)} \leq \delta_0,
%$$}
%if $\mu_1,\mu_2$ are sufficiently large, then the Nash quasi-equilibrium  $(\hat v^1,\hat v^2)$ obtained in Theorem \ref{thm:local_null_controllability} is a Nash equilibrium for the functionals defined in \eqref{eq:def_Ji} associated to the  equation \eqref{eq:PDE}.
%\end{propo}

\noindent{\bf Outline of the paper:}
%In Section \ref{sec:diffeomorphism} we write our problem as an equation with time variable coefficients on a cylindrical domain using a diffeomorphism. 
%In Section \ref{sec:characterization_convexity} we characterize Nash quasi-equilibrium and we prove Proposition \ref{thm:nash_equilibrium} implying the convexity of the functionals and the fact that the Nash quasi-equilibrium are Nash equilibrium. 
In Section \ref{Sec:carleman_nonautonomous} we prove Theorem \ref{prop_carleman_1-intro} using a sequence of lemmas based in the work of Demarque, Límaco and Viana for an autonomous degenerate equation \cite{DemarqueLimacoViana_deg_eq2018}. 
In Section \ref{sec:carleman_system_nonautonomous} we prove a Carleman estimate for the adjoint system of equations associated to \eqref{eq:PDE}. In Section \ref{sec:global_null_controllability}
we prove Theorem \ref{thm:local_null_controllability} which establishes the global null controllability result. 
In Section \ref{sec:final_remarks} we present some remarks and related open problems. 
Finally, in Appendix \ref{appendix A} we show the well-possedness of the system \eqref{eq:PDE}.

%\section{CHARACTERIZATION OF NASH QUASI-EQUILIBRIUM AND EQUILIBRIUM}
%\label{sec:characterization_convexity}

\section{Carleman estimate for a degenerate non-autonomous equation}
\label{Sec:carleman_nonautonomous}

The existence and uniqueness of the system \eqref{eq:PDE} is obtained in the following proposition which is proved in Appendix \ref{appendix A}.

\begin{propo}
Let $u_0,v_0 \in H^1_a(\Omega)$,  $h \in L^2(0,T,L^2(\omega))$. Then the system \eqref{eq:PDE} has a unique solution $(u,w)$ satisfying
\begin{equation*}
\begin{split}
\|(u,v)\|_{L^2(0,T,L^2(\Omega)) \cap H^1(0,T,H^2_a(0,1))} \leq C &\left( \|u_0\|_{H^1_a(\Omega)} + \|v_0\|_{H^1_a(\Omega)} + \|h\|_{L^2(0,T,L^2(\omega))} + \sum_{i=1}^2 \|H_{i}\|_{L^2(0,T,L^2(\Omega))} \right),
\end{split}    
\end{equation*}
where $C=C(\Omega)$ is a positive constant.
\end{propo}
%The proof of this proposition is given in 

\subsection{Hilbert spaces in the divergence case}
\label{subsec:prelim}
   Following \cite{Alabau_cannarsa_fragnelli-06}, for our system which is in divergence form, we take the following weighted Hilbert spaces. In the \emph{weakly degenerate} case, we consider
   $$H^1_a(0,1)=\left\{ u\in L^2(0,1) \ : \ u \ \text{is absolutely continuous in} \ \ (0,1], \sqrt{a}u_x\in L^2(0,1) \ \text{and} \  u(0)=u(1)=0 \right\}$$
   and 
   $$H^2_a(0,1)=\left\{u\in H^1_a(0,1) \ : \ au_x\in H^1(0,1) \right\}$$
   
In both cases, we consider inner products and norms given by
   $$\langle u,v \rangle_{H^1_a(0,1)}=\langle u,v \rangle_{L^2(0,1)}+\langle \sqrt{a}u_x,\sqrt{a}v_x \rangle_{L^2(0,1)}, \ \ \|u\|^2_{H^1_a(0,1)}=\|u\|^2_{L^2(0,1)}+\|\sqrt{a} u_x\|^2_{L^2(0,1)},$$
for all $u, v\in H^1_a(0,1)$, and
  $$\langle u,v \rangle_{H^2_a(0,1)}=\langle u,v \rangle_{H^1_a(0,1)}+\langle ({a}u_x)_x,({a}v_x)_x \rangle_{L^2(0,1)}, \ \ \|u\|^2_{H^2_a(0,1)}=\|u\|^2_{H^1_a(0,1)}+\|({a}u_x)_x\|^2_{L^2(0,1)},$$
for all $u, v\in H^2_a(0,1)$.
In order to get the null controllability of the linear system \eqref{eq:PDE} we prove the observability of the following adjoint system:

\begin{equation}\label{eq:adjoint_PDE}
	\left\{\begin{aligned}
		&-\phi_t - b(t) \left(a(x) \phi_{x}\right)_{x} - (d_1(x,t)\sqrt{a}\phi)_x + b_{11}\phi + b_{21}\psi = F_1 &&\text{in} && Q,\\
        &-\psi_t - b(t) \left(a(x) \psi_{x}\right)_{x} - (d_2(x,t)\sqrt{a}\psi)_x + b_{12}\phi + b_{22}\psi = F_2 &&\text{in}&& Q,\\	&\phi(0,t)=\phi(1,t)=\psi(0,t)=\psi(1,t)=0&&\text{on} && (0,T), \\
		&\phi(\cdot,T) = \phi_T, \quad \psi(\cdot,T) = 0 &&\text{in} && \Omega,
	\end{aligned}
	\right.
\end{equation}
where $\phi_T\in L^2(\Omega)$.

%--------- BEGIN COMMENT ---------
\begin{comment}   
\begin{teo}\label{teo34noauto}
	Assume a (WD) or (SD) at $x_0\in \{0,1\}$ and $b\in W^{1,\infty}(0, T)$ a strictly positive function. If $f\in L^2(0,T; L^2(0,1))$ and $u_0\in L^2(0,1)$, then there exists a unique solution $u\in C([0,T];L^2(0,1))$ of
	(P), and the following inequality holds
	\begin{equation}\label{ec36noauto}
		\sup_{t\in[0,T]}\|u(t)\|^2_{L^2(0,1)}\leq C_T \left(  \|u_0\|^2_{L^2(0,1)}+\|f\|^2_{L^2(0,T;L^2(0,1))} \right)
	\end{equation}
	for a positive constant $C_T$.\\
	In addition, if \ $f\in W^{1,1}(0,T;L^2(0,1))\cap L^2(0,T; H^2_a(0,1))$ and $u_0\in H^2_a(0,1)$ then $u\in C([0,T];H^2_a(0,1))\cap C^1([0,T];L^2(0,1))\subseteq L^2(0,T;H^2_a(0,1))\cap C([0,T];H^1_a(0,1))\cap H^1(0,T;L^2(0,1)) $ and we have
	\begin{eqnarray}
		\sup_{t\in[0,T]}\|u(t)\|^2_{H^1_a(0,1)}\ + \ \int_{0}^{T}\left(  \|u_t\|^2_{L^2(0,1)}+\|(au_x)_x(t)\|^2_{L^2(0,1)} \right)dt\leq C\left( \|u_0\|^2_{H^1_a(0,1)}+\|f\|^2_{L^2(0,T;L^2(0,1))}  \right)
	\end{eqnarray}
\end{teo}
este resultado hay q adecuarle para el sistema de ecuaciones que se tiene del sistema optimal linealizado 
\end{comment}
%------------------ END COMMENT -------------------------
   
\subsection{Weights of Carleman estimate}
%%%%%% Do artigo HIERARQUICO 2 %%%%%%%
%Since $O_d \cap O \neq \emptyset$, then it was proved in \cite{Fur_Ima-96} that there is a function $\sigma \in C^2([0,1])$ such that, in an open set $O' = (\alpha',\beta') \subset O \cap O_d$, 
%$$
%\sigma(x) > 0 \quad \text{in} \quad (0,1), \quad \sigma(0)=\sigma(1) = 0, \quad \sigma_x(x) \neq 0, \quad \text{in} \quad [0,1]-O_o,
%$$
%where $O_o \subset\subset O'$ is an open subset.

Following \cite{DemarqueLimacoViana_deg_eq2018, DemarqueLimacoViana_deg_sys2020}, let $\omega'=(\alpha',\beta') \subset\subset \omega$ and let $\Psi : [0,1] \to \R$ be a $C^2$ function such that
$$
\Psi(x) = 
\begin{cases}
\int_0^x \frac{s}{a(s)} ds, \quad x \in [0,\alpha'), \\
-\int_{\beta'}^x \frac{s}{a(s)} ds, \quad x \in [\beta',1].
\end{cases}
$$
For $\lambda \geq \lambda_0$ define the functions
$$
\theta(t) = \frac{1}{(t(T-t))^4}, \qquad \eta(x) = e^{\lambda(|\Psi|_\infty + \Psi)}, \qquad \sigma(x,t) = \theta(t) \eta(x),
$$
$$
\varphi(x,t) = \theta(t) (e^{\lambda(|\Psi|_\infty + \Psi)}-e^{3\lambda|\Psi|_\infty}).
$$

%\begin{equation}\label{eq:PDE}
%	\left\{\begin{aligned}
%		&u_t - b(t) \left(a(x) u_{x}\right)_{x} + d_1(x,t)u_x + b_{11}u + b_{12}v = h 1_{O} + H_1 &&\text{in} && Q,\\
%        &v_t - b(t) \left(a(x) v_{x}\right)_{x} + d_2(x,t)v_x + b_{21}u + b_{22}v = H_2 &&\text{in}&& Q,\\
%		&u(0,t)=u(1,t)=v(0,t)=v(1,t)=0&&\text{on} && (0,T), \\
%		&u(\cdot,0) = u_0, \quad v(\cdot,0) = v_0 &&\text{in} && \Omega_0,
%	\end{aligned}
%	\right.
%\end{equation}

We introduce the notation
$$
I(\varsigma) = \int_0^T\int_\Omega e^{2s\varphi}((s\lambda)\sigma b(t)^2 a(x) \varsigma_x^2 + (s\lambda)^2 \sigma^2 b(t)^2\varsigma^2)dxdt.
$$

%The following proposition was proved for an autonomous system in a cylindrical domain \cite{DemarqueLimacoViana_deg_sys2020} for an adjoint system of two retrograde equations. 

%After making the substitution $t \mapsto T-t$ to our second equation in \eqref{eq:adjoint_PDE}, we are in the case treated in \cite{DemarqueLimacoViana_deg_sys2020} and the conclusion remains true after coming back to the original variable.

%For a non-autonomous equation we get an analogous result. 

First we show a Carleman estimate for the linear equation
   \begin{equation}\label{adjunto2_texto_art}
	\begin{cases}
		w_t+b(t)\left(a(x)w_x\right)_x+d(x,t)\sqrt{a}w_x + c(x,t)w  = h(x,t), & \ \ \ \text{en} \ \ \ {Q}\\
		w=0, & \ \ \ \text{in} \ \ \ {\Sigma}\\
        w(T) = w_T,& \ \ \ \text{in} \ \ \ (0,1)
	\end{cases}
\end{equation}
where $b$ belongs to $W^{1,\infty}(0,T)$, $b(t) \geq b_0 >0$ and $\frac{b'(t)}{b(t)} \leq B$ for some constants $b_0, B>0$. 
%where $b$ is a continuous function in $[0,T]$ such that $\frac{b'(t)}{b(t)} \leq B$, for some constant $B>0$. 

\begin{teo}
\label{prop_carleman_1}
	There exist $C>0$ and $\lambda_0,  s_0>0$ such that every solution $w$ of (\ref{adjunto2_texto_art}) satisfies, for all
	$s\geq s_0$ and $\lambda\geq \lambda_0$,
	\begin{equation}\label{ecjv34}
		%\int_{0}^{T}\int_{0}^{1}e^{2s\varphi}\left(  (s\lambda)\sigma a b^2 w_x^2+(s\lambda)^2\sigma^2b^2 w^2 \right)
        I(w) \leq C\left(	\int_{0}^{T}\int_{0}^{1} e^{2s\varphi} |h|^2+(\lambda s)^3\int_{0}^{T}\int_{\omega}e^{2s\varphi} \sigma^3  w^2  \right).
	\end{equation}
\end{teo}
%The proof of this proposition is presented in Appendix \ref{appendix B}.

\subsection{Proof of Carleman estimate}

Let us first consider a degenerate equation of the form:
   \begin{equation}\label{adjunto2}
	\begin{cases}
            v_t+b(t)\left(a(x)v_x\right)_x =h(x,t), & \ \ \   \text{in} \ \ \ {Q},\\
		v=0, & \ \ \ \text{on} \ \ \ {\Sigma}, \\
            v(T)=v_T, & \ \ \ \text{in} \ \ \ {(0,1)}.
	\end{cases}
\end{equation}
%Recall that we suppose that $b$ is a continuous function in $[0,T]$ 
and we denote
$$
m := \min\limits_{t \in [0,T]} b(t), \qquad M :=\max\limits_{t \in [0,T]} b(t).
$$
%Moreover we suppose that $\frac{b'(t)}{b(t)} \leq B$, for some constant $B>0$. 

%We recall some functions and notations for the weights. Let $\omega'=(\alpha',\beta')\subset \subset \omega$ and let   $\psi : [0,1]\rightarrow \rea$ be a $C^2$ function such that
%\begin{equation}\label{ecjv32}
%	\psi(x)=\begin{cases}
%		\int\limits_{0}^{x}\frac{y}{a(y)}dy, \ \ & \ \ x\in [0,\alpha')\\
%		-\int\limits_{\beta'}^{x}\frac{y}{a(y)}dy, \ \ & \ \ x\in [\beta',1]\\
%	\end{cases}
%\end{equation}
%For $\lambda\geq \lambda_0$ define
%\begin{equation}\label{ecjv33}
%	\theta(t)=\frac{1}{t^4(T-t)^4}, \quad \eta(x)=e^{\lambda(|\psi|_\infty+\psi)}, \quad  \sigma(x,t)=\eta(x)\theta(t), \quad  \varphi(x,t)=\theta(t)\left( e^{\lambda(|\psi|_\infty+\psi)}-e^{3\lambda|\psi|_\infty} \right)
%\end{equation}

\begin{propo}\label{propojva1}
\label{prop_carleman_11}
	There exist $C>0$ and $\lambda_0,  s_0>0$ such that every solution $v$ of (\ref{adjunto2}) satisfies, for all
	$s\geq s_0$ and $\lambda\geq \lambda_0$,
	\begin{equation}\label{ecjv34}
		\int_{0}^{T}\int_{0}^{1}e^{2s\varphi}\left(  (s\lambda)\sigma a(x) b(t)^2 v_x^2+(s\lambda)^2\sigma^2b(t)^2 v^2 \right)\leq C\left(	\int_{0}^{T}\int_{0}^{1} e^{2s\varphi} |h|^2+(\lambda s)^3\int_{0}^{T}\int_{\omega}e^{2s\varphi} \sigma^3  v^2  \right)
	\end{equation}
	
\end{propo}
As usual, the proof of Proposition \ref{propojva1} relies on the change of variables $w=e^{2s\varphi}v$. Notice that
\begin{eqnarray*}
	v_t&=&e^{-s\varphi}(-s\varphi_t w+w_t),\\
	b(t)\left(a(x)v_x\right)_x&=&e^{-s\varphi}(b(a w_x)_x-sb(a\varphi_x)_x w-2sab\varphi_x w_x+s^2ab\varphi^2_x w).
\end{eqnarray*}

Then, from (\ref{adjunto2}), we obtain
\begin{equation}\label{adjunto22}
	\begin{cases}
		L^+ w+L^- w=e^{s\varphi}h, & \ \ \ \text{en} \ \ \ {Q}\\
		w=0, & \ \ \ \text{en} \ \ \ {\Sigma}\\
		w(x,0)=w(x,T)=0
	\end{cases}
\end{equation}
where
\begin{eqnarray*}
	L^+ w&=&b(a w_x)_x-s\varphi_t w+s^2ab\varphi^2_x w\\
	L^- w&=&w_t-2sab\varphi_x w_x-sb(a\varphi_x)_x w.
\end{eqnarray*}

Thus,
$$\|L^+ w\|^2+\|L^- w\|^2+2\langle L^+ w,L^- w \rangle=\|e^{s\varphi} h\|^2,$$
where $\|\cdot\|$ and $\langle \cdot,\cdot\rangle$ denote the norm and the inner product in $L^2(Q)$, respectively. The proof of our Carleman inequality follows using Lemmas \ref{lemaA1}-\ref{lemaA10}.

In the following inequalities $C$ denotes an arbitrary positive constant whose value could change from line to line. The dot in $\dot b$ means the derivative with respect to $t$.
\begin{lema}\label{lemaA1}	
    \begin{align*}
    \langle L^+ w,L^- w \rangle = &\frac{s}  {2}\int_{Q}\varphi_{tt}w^2+\frac{1}{2}\int_{Q}a\dot{b}w_x^2-\frac{s^2}{2}\int_{Q}a\dot{b} \varphi_x^2 w^2+s\int_{Q}ab^2 (a\varphi_x)_{xx}w w_x+s\int_{Q}ab^2(2(a \varphi_{x})_x - a_x \varphi_x )w_x^2\\
    &-2s^2\int_{Q}ab \varphi_x \varphi_{xt}w^2 + s^3\int_{Q}  ab^2 \varphi_x (a\varphi_x^2)_x w^2-s\int_{0}^{T}\left[ a^2\varphi_xw_x^2  \right]_{x=0}^{x=1}.
    \end{align*}
\end{lema}

\begin{proof}
The idea is to express the inner product using integrals with integrands $w^2$ and $w_x^2$.
%in order to use one more time the change of variables defined above.
%la idea es expresar el producto interior en funcion de las integrales con terminos $w^2$, $w_x^2$ para volver utilizar el cambio de variable ya antes definida
\begin{align*}
    \langle L^+ w,L^- w \rangle= &\int_{Q}\left( b(a w_x)_x-s\varphi_t w+s^2ab\varphi^2_x w  \right)w_t+s^2\int_{Q} \varphi_t w \left( 2ab\varphi_x w_x+b(a\varphi_x)_x w  \right)\\
	&-s^3\int_{Q}  ab\varphi_x^2w\left( 2ab\varphi_x w_x+b(a\varphi_x)_x w  \right) -s\int_{Q}  b(a w_x)_x \left( 2ab\varphi_x w_x+b(a\varphi_x)_x w  \right) \\
    = &I_1 + I_2 + I_3 + I_4.
\end{align*}

For $I_1$, performing integration by parts in space and using the identities $w_xw_{xt}=\frac{1}{2}(w_x^2)_t, \ \  w w_t=\frac{1}{2}(w^2)_t$ we get
%utilizando integracion por partes en el espacio en el primer sumando y las identidades en los otros:  $w_xw_{xt}=\frac{1}{2}(w_x^2)_t, \ \  w w_t=\frac{1}{2}(w^2)_t$
$$I_1=\int_{Q}\left( b(a w_x)_x-s\varphi_t w+s^2ab\varphi^2_x w  \right)w_t=\int_{0}^{T}\left[ab w_xw_t\right]^{x=1}_{x=0}-\frac{1}{2}\int_{Q}ab (w_x^2)_{_t} -\frac{s}{2}\int_{Q} \varphi_t(w^2)_t+\frac{s^2}{2}\int_{Q}ab\varphi_x^2(w^2)_t. $$

Integrating by parts one more time and evaluating the boundary conditions we get
%integracion por partes en el tiempo 
%hspace*{-1.1cm} I_1=\int_{0}^{T}\left[ab w_xw_t\right]^{x=1}_{x=0}-\frac{1}{2}\int_{0}^{1}\left[abw_x^2\right]^{t=0}_{t=T}+\frac{1}{2}\int_{Q}a\dot{b} w_x^2-\frac{s}{2}\int_{0}^{1}\left[\varphi_tw^2\right]^{t=0}_{t=T}+\frac{s}{2}\int_{Q}\varphi_{tt} w^2+\frac{s^2}{2}\int_{0}^{1}\left[ab\varphi_x^2 w^2\right]^{t=0}_{t=T}-\frac{s^2}{2}\int_{Q}a(b\varphi_{x}^2)_t w^2$$
\begin{align*}
    I_1&=\int_{0}^{T}\left[ab w_xw_t\right]^{x=1}_{x=0}-\frac{1}{2}\int_{0}^{1}\left[abw_x^2\right]^{t=0}_{t=T}+\frac{1}{2}\int_{Q}a\dot{b} w_x^2+\frac{s}{2}\int_{Q}\varphi_{tt} w^2-\frac{s^2}{2}\int_{Q}a(b\varphi_{x}^2)_t w^2 \\
    &=\frac{1}{2}\int_{Q}a\dot{b} w_x^2 + \frac{s}{2}\int_{Q}\varphi_{tt} w^2-\frac{s^2}{2}\int_{Q}a\dot{b}\varphi_{x}^2 w^2+s^2\int_{Q}ab\varphi_{xt}\varphi_{x} w^2.    
\end{align*}

%$$I_1=\int_{0}^{T}\left[ab w_xw_t\right]^{x=1}_{x=0}-\frac{1}{2}\int_{0}^{1}\left[abw_x^2\right]^{t=0}_{t=T}+\frac{1}{2}\int_{Q}a\dot{b} w_x^2+\frac{s}{2}\int_{Q}\varphi_{tt} w^2-\frac{s^2}{2}\int_{Q}a(b\varphi_{x}^2)_t w^2$$
%$$\hspace*{-1.1cm} I_1=\frac{1}{2}\int_{Q}a\dot{b} w_x^2-\frac{s}{2}\int_{Q}\varphi_{tt} w^2-\frac{s^2}{2}\int_{Q}a(b\varphi_{x}^2)_t w^2$$
$$$$

For $I_2$ we have:
%$$I_2=s^2\int_{Q} \varphi_t w \left( 2ab\varphi_x w_x+b(a\varphi_x)_x w  \right)$$
\begin{align*}
	\int_{Q} \varphi_t w \left( 2ab\varphi_x w_x+b(a\varphi_x)_x w  \right)
	&=\int_{Q}2ab\varphi_t\varphi_x w_x w+\int_{0}^{T}\left[b\varphi_t  w^2a\varphi_x\right]^{x=1}_{x=0}-\int_{Q}(b\varphi_t  w^2)_x a\varphi_x\\
    &=\int_{Q}2ab\varphi_t\varphi_x w_x w-\int_{Q}b\varphi_{xt}  w^2 a\varphi_x-\int_{Q}b\varphi_{t}  2w w_x a\varphi_x.
\end{align*}

Thus, 
$$I_2=-s^2\int_{Q} ab\varphi_{xt}\varphi_x  w^2.$$

Similarly, for $I_3$,
%$$I_3=-s^3\int_{Q}  ab\varphi_x^2w\left( 2ab\varphi_x w_x+b(a\varphi_x)_x w  \right)$$
\begin{align*}
    \int_{Q}  ab\varphi_x^2w\left( 2ab\varphi_x w_x+b(a\varphi_x)_x w  \right) &=
    %\int_{Q} 2a^2b^2\varphi_x^3w_x w+\int_{Q} ab^2\varphi_x^2(a\varphi_x)_x w^2=
    \int_{Q} 2a^2b^2\varphi_x^3w_x w-\int_{0}^{T}\left[a^2b^2\varphi_x^2\varphi_x w^2\right]^{x=1}_{x=0}
	-\int_{Q} b^2 (a\varphi_x^2  w^2)_x a\varphi_x \\
    &=\int_{Q} 2a^2b^2\varphi_x^3w_x w-\int_{Q} b^2 a\varphi_x^2(  w^2)_x a\varphi_x-\int_{Q} ab^2\varphi_x (a\varphi_x^2)_x  w^2.
\end{align*}

Thus,
$$I_3=s^3\int_{Q} ab^2\varphi_x (a\varphi_x^2)_x  w^2$$

Finally, for $I_4$,
%$$I_4=-s\int_{Q}  b(a w_x)_x \left( 2ab\varphi_x w_x+b(a\varphi_x)_x w  \right) $$
\begin{multline*}
	\int_{Q}  b(a w_x)_x \left( 2ab\varphi_x w_x+b(a\varphi_x)_x w  \right) =\int_{Q}2ab^2\varphi_x w_x(a w_x)_x+\int_{Q}b^2(a\varphi_x)_x(a w_x)_x w \\
    %=\int_{0}^{T}\left[2a^2b^2\varphi_x w_x^2\right]^{x=1}_{x=0} -\int_{Q}2ab^2(a\varphi_x w_x)_xw_x+\int_{0}^{T}\left[b^2w(a\varphi_x)_x a w_x\right]^{x=1}_{x=0}-\int_{Q}b^2(w(a\varphi_x)_x)_xa w_x\\
	=\int_{0}^{T}\left[2a^2b^2\varphi_x w_x^2\right]^{x=1}_{x=0}-\int_{Q}2ab^2(a\varphi_x w_x)_xw_x-\int_{Q}b^2(w(a\varphi_x)_x)_xa w_x\\
	%=\int_{0}^{T}\left[2a^2b^2\varphi_x w_x^2\right]^{x=1}_{x=0}-\int_{Q}2ab^2(a\varphi_x )_xw_x^2-\int_{Q}2a^2b^2\varphi_x w_{xx} w_x-\int_{Q}ab^2(a\varphi_x)_x w_x^2-\int_{Q}ab^2(a\varphi_x)_{xx} w_x w\\
	=\int_{0}^{T}\left[2a^2b^2\varphi_x w_x^2\right]^{x=1}_{x=0}-\int_{Q}3ab^2(a\varphi_x )_xw_x^2-\int_{Q}a^2b^2\varphi_x (w_{x}^2)_x -\int_{Q}ab^2(a\varphi_x)_{xx} w_x w\\
	=\int_{0}^{T}\left[2a^2b^2\varphi_x w_x^2\right]^{x=1}_{x=0}-\int_{Q}3ab^2(a\varphi_x )_xw_x^2-\int_{0}^{T}\left[a^2b^2\varphi_x w_x^2\right]^{x=1}_{x=0}+\int_{Q}b^2 (a^2\varphi_x)_xw_{x}^2 -\int_{Q}ab^2(a\varphi_x)_{xx} w_x w
	\\
	=\int_{0}^{T}\left[a^2b^2\varphi_x w_x^2\right]^{x=1}_{x=0}-\int_{Q}3ab^2(a\varphi_x )_xw_x^2+\int_{Q}b^2 a(a\varphi_x)_xw_{x}^2 +\int_{Q}b^2 a_x a\varphi_xw_{x}^2-\int_{Q}ab^2(a\varphi_x)_{xx} w_x w
	\\
	=\int_{0}^{T}\left[a^2b^2\varphi_x w_x^2\right]^{x=1}_{x=0}-\int_{Q}2ab^2(a\varphi_x )_xw_x^2 +\int_{Q}b^2 a_x a\varphi_xw_{x}^2-\int_{Q}ab^2(a\varphi_x)_{xx} w_x w
\end{multline*}
Then,
$$I_4=s\int_{Q}ab^2(a\varphi_x)_{xx} w_x w+2s\int_{Q}ab^2(a\varphi_x )_xw_x^2-s\int_{Q}ab^2 a_x \varphi_xw_{x}^2-s\int_{0}^{T}\left[a^2b^2\varphi_x w_x^2\right]^{x=1}_{x=0}.$$
\end{proof}

\begin{lema}[For $I_1$]\label{lemaA2}
%$$\hspace*{-1.1cm} I_1=\frac{1}{2}\int_{Q}a\dot{b} w_x^2-\frac{s}{2}\int_{Q}\varphi_{tt} w^2-\frac{s^2}{2}\int_{Q}a\dot{b}\varphi_{x}^2 w^2-s^2\int_{Q}ab\varphi_{xt}\varphi_{x} w^2$$

$$- \frac{s^2}{2} \int_{Q}a\dot{b}\varphi_{x}^2 w^2 \geq - Cs^2\lambda^2\left[ \int_0^T \int_0^{\alpha'} b^2 \frac{x^2}{a} \sigma^3 w^2 + \int_0^T\int_{\omega'} b^2 \sigma^3 w^2 + \int_Q b^2 a^2 |\psi_x|^4 \sigma^3 w^2   \right]$$
\end{lema}	

\begin{proof}
Using the assumption $\left|\frac{\dot{b}(t)}{b(t)} \right|\leq C$, the property $\varphi_x=\lambda \psi_x \sigma$ and the bounds  $m=\min b(t)$, $M=\max b(t)$,  
%	Asumiendo que $\left|\frac{\dot{b}(t)}{b(t)} \right|\leq C$, sabemos que $\varphi_x=\lambda \psi_x \sigma$ y $m=\min b(t)$, $M=\max b(t)$
\begin{eqnarray*}
	\left|- \frac{1}{2} \int_{Q}a\dot{b}\varphi_{x}^2 w^2 \right|&\leq &\frac{1}{2}\int_{Q}a|\dot{b}||\varphi_{x}|^2 |w|^2\leq  C\int_{Q}a|b||\varphi_{x}|^2 |w|^2\leq  \frac{C}{m}\int_{Q}a|b|^2|\varphi_{x}|^2 |w|^2\\
	\left|- \frac{1}{2} \int_{Q}a\dot{b}\varphi_{x}^2 w^2 \right|&\leq & C\int_{Q}\lambda^2a|b|^2|\psi_{x}|^2 \sigma^2|w|^2\\
		\left|- \frac{1}{2} \int_{Q}a\dot{b}\varphi_{x}^2 w^2 \right|&\leq &C\lambda^2\left[ \int_{0}^{T} \int_{0}^{\alpha'} ab^2 |\psi_x|^2 \sigma^3 w^2 +  \int_{0}^{T} \int_{\omega'}ab^2 |\psi_x|^2 \sigma^3 w^2+\int_{0}^{T} \int_{\beta'}^{1}ab^2 |\psi_x|^2 \sigma^3 w^2   \right]
\end{eqnarray*}
We recall that  $a|\psi_x|^2$ is bounded in $\omega'$ and $\psi_x=\frac{x}{a}, \ \ x\in [0,\alpha']$ and  $a|\psi_x|^2\leq a^2|\psi_x|^4$ for $x\in[\beta',1]$. Thus,
%		donde sabemos que $a|\psi_x|^2$ es limitado en $\omega'$ y tambien tenemos que $\psi_x=\frac{x}{a}, \ \ x\in [0,\alpha']$ y \\ $a|\psi_x|^2\leq a^2|\psi_x|^4$ en $x\in[\beta',1]$
		\begin{eqnarray*}
			\left|- \frac{s^2}{2} \int_{Q}a\dot{b}\varphi_{x}^2 w^2 \right|&\leq &Cs^2\lambda^2\left[ \int_0^T \int_0^{\alpha'} b^2 \frac{x^2}{a} \sigma^3 w^2 + \int_0^T\int_{\omega'} b^2 \sigma^3 w^2 + \int_Q b^2 a^2 |\psi_x|^4 \sigma^3 w^2   \right].
	\end{eqnarray*}	
%    This implies the conclusion.
%$$- \frac{s^2}{2} \int_{Q}a\dot{b}\varphi_{x}^2 w^2 \geq - Cs^2\lambda^2\left[ \int_0^T \int_0^{\alpha'} b^2 \frac{x^2}{a} \sigma^3 w^2 + \int_0^T\int_{\omega'} b^2 \sigma^3 w^2 + \int_Q b^2 a^2 |\psi_x|^4 \sigma^3 w^2   \right]$$
\end{proof}

\begin{lema}[For $I_1$]\label{lemaA2_1}
$$\frac{1}{2}\int_{Q}a\dot{b} w_x^2 \geq -C\left[ \int_{0}^{T} \int_{0}^{\alpha'} ab^2\sigma w_x^2+  \int_{0}^{T} \int_{\omega'}b^2\sigma w_x^2+\int_{0}^{T} \int_{\beta'}^{1}a^2b^2\psi_{x}^2 \sigma w_x^2  \right]$$
\end{lema}
\begin{proof}
Using analogous estimates as in Lemma \ref{lemaA2} and Young's inequality, we have 
	\begin{eqnarray*}
		\left| \frac{1}{2}\int_{Q}a\dot{b} w_x^2 \right|&\leq &\frac{1}{2}\int_{Q}a|\dot{b}| |w_x|^2\leq  C\int_{Q}a|b| |w_x|^2\leq  \frac{C}{m}\int_{Q}a|b|^2|w_x|^2\\
		\left|\frac{1}{2}\int_{Q}a\dot{b} w_x^2 \right|&\leq & C\int_{Q}\left(a^{1/2}b\sigma^{1/2} w_x  \right)\left( a^{1/2}b\sigma^{-1/2} w_x  \right)\leq C\left[\int_{Q}ab^2\sigma w_x^2+\int_{Q}ab^2\sigma^{-1} w_x^2   \right].
	\end{eqnarray*}	
    
Since $\theta(t)=\frac{1}{(t(T-t))^4}$, the minimum value is attained in $t=T/2$, therefore we have that $\theta(t)\geq \theta(T/2)=\left(\frac{4}{T^2}\right)^4$. By definition  $\eta(x)=e^{\lambda(|\psi|_\infty+\psi)}\geq 1$ then $\sigma(x,t)=\theta(t)\eta(x)\geq \theta(t)\geq \left(\frac{4}{T^2}\right)^4$.
%	Como $\theta(t)=\frac{1}{(t(T-t))^4}$ tenemos que el valor minimo alcanza en $t=T/2$ por tanto se tiene que\\ $\theta(t)\geq \theta(T/2)=\left(\frac{4}{T^2}\right)^4$, por definicion $\eta(x)=e^{\lambda(|\psi|_\infty+\psi)}\geq 1$ luego $\sigma(x,t)=\theta(t)\eta(x)\geq \theta(t)\geq \left(\frac{4}{T^2}\right)^4$\\
	$\sigma^{-1}\leq \left(\frac{T^2}{4}\right)^4$
	\begin{eqnarray*}
		\left|\frac{1}{2}\int_{Q}a\dot{b} w_x^2 \right|&\leq & C\left[\int_{Q}ab^2\sigma w_x^2+\int_{Q}ab^2\sigma^{-1} w_x^2   \right]\leq  C\left[\int_{Q}ab^2\sigma w_x^2+\left(\frac{T^2}{4}\right)^4\int_{Q}ab^2 w_x^2   \right]
\end{eqnarray*}		
$$\left|\frac{1}{2}\int_{Q}a\dot{b} w_x^2 \right|\leq C \int_{Q}ab^2\sigma w_x^2= C\left[ \int_{0}^{T} \int_{0}^{\alpha'} ab^2\sigma w_x^2+  \int_{0}^{T} \int_{\omega'}ab^2\sigma w_x^2+\int_{0}^{T} \int_{\beta'}^{1}ab^2\sigma w_x^2  \right]$$
On the other hand we have that $a\leq Ca^2\psi_{x}^2$ for $x\in (\beta',1)$, then
%tenemos que $a\leq Ca^2\psi_{x}^2$ en $x\in (\beta',1)$
$$\left|\frac{1}{2}\int_{Q}a\dot{b} w_x^2 \right|\leq C\left[ \int_{0}^{T} \int_{0}^{\alpha'} ab^2\sigma w_x^2+  \int_{0}^{T} \int_{\omega'}b^2\sigma w_x^2+\int_{0}^{T} \int_{\beta'}^{1}a^2b^2\psi_{x}^2 \sigma w_x^2  \right].$$
%$$\frac{1}{2}\int_{Q}a\dot{b} w_x^2 \geq -C\left[ \int_{0}^{T} \int_{0}^{\alpha'} ab^2\sigma w_x^2+  \int_{0}^{T} \int_{\omega'}b^2\sigma w_x^2+\int_{0}^{T} \int_{\beta'}^{1}a^2b^2\psi_{x}^2 \sigma w_x^2  \right]$$    
\end{proof}

\begin{lema}[For $I_1$]\label{lemaA2_2}
\begin{align*}
    \frac{s}{2}\int_{Q}\varphi_{tt} w^2 \geq -C &\left[  s^{1/2}\int_0^T \int_0^{\alpha'} ab^2\sigma  w_x^2 + s^{1/2}\int_0^T\int_{\omega'} b^2 \sigma w_x^2 + s^{1/2}\int_{0}^{T} \int_{\beta'}^{1} b^2 a^2 |\psi_x|^2 \sigma w_x^2\right.\\
	&\left.+s^2\lambda^2\int_0^T \int_0^{\alpha'} b^2 \frac{x^2}{a} \sigma^3 w^2 + s^2\lambda^2\int_0^T\int_{\omega'} b^2 \sigma^3 w^2 +s^2\lambda^2 \int_Q b^2 a^2 |\psi_x|^4 \sigma^3 w^2\right].
\end{align*}
\end{lema}

\begin{proof}
%Estimate for the second integral, 
Using that $|\varphi_{tt}|\leq C\sigma^{3/2}$:
%Estimativa para la segunda integral, sabiendo que $|\varphi_{tt}|\leq C\sigma^{3/2}$
\begin{align*}
    \left|-\frac{s}{2}\int_{Q}\varphi_{tt} w^2 \right| &\leq \frac{s}{2}\int_{Q}|\varphi_{tt}| |w|^2\leq \frac{s}{2m^2}\int_{Q}|\varphi_{tt}||b|^2 |w|^2 \leq Cs\int_{Q}\sigma^{3/2}|b|^2 |w|^2 \\
    &\leq C\int_{Q} \left(s^{1/4}b\sigma^{1/2}\frac{a^{1/2}}{x} w  \right)\left( s^{3/4}b\sigma^{}\frac{x}{a^{1/2}} w \right)
\leq C\int_{Q}s^{1/2}b^2\sigma\frac{a}{x^2} w+C\int_{Q}s^{3/2}b^2\sigma^2\frac{x^2}{a} w.
\end{align*}
%$$\left|-\frac{s}{2}\int_{Q}\varphi_{tt} w^2 \right|\leq \frac{s}{2}\int_{Q}|\varphi_{tt}| |w|^2\leq \frac{s}{2m^2}\int_{Q}|\varphi_{tt}||b|^2 |w|^2 \leq Cs\int_{Q}\sigma^{3/2}|b|^2 |w|^2$$
%	\begin{eqnarray*}
%	\left|-\frac{s}{2}\int_{Q}\varphi_{tt} w^2 \right|&\leq & C\int_{Q} \left(s^{1/4}b\sigma^{1/2}\frac{a^{1/2}}{x} w  \right)\left( s^{3/4}b\sigma^{}\frac{x}{a^{1/2}} w \right)\\
%	&\leq &C\int_{Q}s^{1/2}b^2\sigma\frac{a}{x^2} w+C\int_{Q}s^{3/2}b^2\sigma^2\frac{x^2}{a} w
%\end{eqnarray*}	

Applying the Hardy-Poincaré inequality, 
%Aplicando la desigualdad de Hardy-Poincare
	\begin{eqnarray*}
	\left|-\frac{s}{2}\int_{Q}\varphi_{tt} w^2 \right| \leq Cs^{1/2}\int_{Q}ab^2\sigma  w_x^2+Cs^{3/2}\int_{Q}b^2\sigma^2\frac{x^2}{a} w
	%	\left|-\frac{s}{2}\int_{Q}\varphi_{tt} w^2 \right|&
        \leq Cs^{1/2}\int_{Q}ab^2\sigma  w_x^2+Cs^2\lambda^2\int_{Q}b^2\sigma^3\frac{x^2}{a} w.
\end{eqnarray*}	

Decomposing the spatial domain into the intervals $[0,\alpha'], \omega' $ and $[\beta',1]$ and taking into account the relations
%Descomponiendo el dominio espacial en los intervalos $[0,\alpha'], \omega' $ y $[\beta',1]$ y tomando en cuenta las relaciones 
$$a\leq C a^2 |\psi_x|^2,  \ \ \frac{x^2}{a}\leq Ca^2|\psi_x|^4, \ \ x\in[\beta',1], $$
we get our estimate:
%obtenemos la estimativa para la ultima integral 
\begin{multline*}
	\left|-\frac{s}{2}\int_{Q}\varphi_{tt} w^2 \right|\leq C\left[ s^{1/2}\int_0^T \int_0^{\alpha'} ab^2\sigma  w_x^2 + s^{1/2}\int_0^T\int_{\omega'} b^2 \sigma w_x^2 + s^{1/2}\int_{0}^{T} \int_{\beta'}^{1} b^2 a^2 |\psi_x|^2 \sigma w_x^2\right.\\
	\left.+s^2\lambda^2\int_0^T \int_0^{\alpha'} b^2 \frac{x^2}{a} \sigma^3 w^2 + s^2\lambda^2\int_0^T\int_{\omega'} b^2 \sigma^3 w^2 +s^2\lambda^2 \int_Q b^2 a^2 |\psi_x|^4 \sigma^3 w^2\right].
\end{multline*}
\end{proof}

\begin{lema}[for $I_2$]\label{lemaA3}
	$$I_2=-s^2\int_{Q} ab\varphi_{xt}\varphi_x  w^2 \geq
	- C s^2 \lambda^2 \left( \int_0^T \int_0^{\alpha'} b^2 \frac{x^2}{a} \sigma^3 w^2 + \int_0^T\int_{\omega'} b^2 \sigma^3 w^2 + \int_Q b^2 a^2 |\psi_x|^4 \sigma^3 w^2 \right).$$
\end{lema}	
\begin{proof}
    We know that $\varphi_x=\theta \lambda \psi_x \eta$ and $\varphi_{xt}=\dot{\theta}\lambda \psi_x \eta$. Moreover, we have $m \leq b(t)$, where $m$ is the minimum value of $b(t)$. 

	\begin{eqnarray*}
		\left| -s^2\int_{Q} ab\varphi_{xt}\varphi_x  w^2 \right|	&\leq &s^2\lambda^2 \int_{Q} ab|\theta \dot{\theta}| |\psi_x|^2 \eta^2 w^2\\
		&\leq & \frac{C}{m}s^2\lambda^2 \int_{Q} ab^2 |\psi_x|^2 \sigma^3 w^2\\
		&\leq & Cs^2\lambda^2\left[ \int_{0}^{T} \int_{0}^{\alpha'} ab^2 |\psi_x|^2 \sigma^3 w^2 +  \int_{0}^{T} \int_{\omega'}ab^2 |\psi_x|^2 \sigma^3 w^2+\int_{0}^{T} \int_{\beta'}^{1}ab^2 |\psi_x|^2 \sigma^3 w^2 \right]
	\end{eqnarray*}
    
	Since $a|\psi_x|^2$ is bounded in $\omega'$ and we have $\psi_x=\frac{x}{a}, \ \ x\in [0,\alpha']$ and  $a|\psi_x|^2\leq a^2|\psi_x|^4$ in $x\in[\beta',1]$, 
	\begin{eqnarray*}
		\left| -s^2\int_{Q} ab\varphi_{xt}\varphi_x  w^2 \right|	&\leq & Cs^2\lambda^2\left[ \int_{0}^{T} \int_{0}^{\alpha'} b^2 \frac{x^2}{a} \sigma^3 w^2 +  \int_{0}^{T} \int_{\omega'}b^2  \sigma^3 w^2+\int_{0}^{T} \int_{\beta'}^{1}a^2b^2 |\psi_x|^4 \sigma^3 w^2 \right].
	\end{eqnarray*}
%	$$-s^2\int_{Q} ab\varphi_{xt}\varphi_x  w^2 	\geq  -Cs^2\lambda^2\left[ \int_{0}^{T} \int_{0}^{\alpha'} b^2 \frac{x^2}{a} \sigma^3 w^2 +  \int_{0}^{T} \int_{\omega'}b^2  \sigma^3 w^2+\int_{0}^{T} \int_{\beta'}^{1}a^2b^2 |\psi_x|^4 \sigma^3 w^2 \right]$$
\end{proof}

\begin{lema}[For $I_3$]\label{lemaA4}
	$$I_3=s^3\int_{Q} ab^2\varphi_x (a\varphi_x^2)_x  w^2\geq
	C s^3 \lambda^3 \int_0^T \int_0^{\alpha'} b^2 \frac{x^2}{a} \sigma^3 w^2 - Cs^3\lambda^3 \int_0^T\int_{\omega'} b^2 \sigma^3 w^2 + C s^3 \lambda^4 \int_Q b^2 a^2 |\psi_x|^4 \sigma^3 w^2.$$
\end{lema}
\begin{proof}
	\begin{eqnarray*}
		\int_{Q} ab^2\varphi_x (a\varphi_x^2)_x  w^2&=&\int_{Q} ab^2 \varphi_x (a_x \varphi_x^2+2a\varphi_x\varphi_{xx})  w^2=\int_{Q} ab^2 \lambda\psi_x\sigma (a_x \lambda^2\psi_x^2\sigma^2+2a\lambda\psi_x\sigma(\lambda \psi_{xx}\sigma+\lambda^2\psi_x^2\sigma))  w^2\\
		&=&\int_{Q} ab^2  (\lambda^3\sigma^3a_x \psi_x^3+2a\lambda^3\sigma^3\psi_x^2\psi_{xx}+2a\lambda^4\sigma^3\psi_x^4)  w^2\\
		&=&\int_{Q} ab^2  \lambda^3\sigma^3(a_x \psi_x^3+2a\psi_x^2\psi_{xx})w^2+\int_{Q} 2a^2b^2\lambda^4\sigma^3\psi_x^4  w^2\\
		&=&\lambda^3\int_{Q} ab^2 \psi_x (a\psi_x^2)_x\sigma^3w^2+2\lambda^4\int_{Q} a^2b^2\sigma^3\psi_x^4  w^2
	\end{eqnarray*}

    For the first integral in the right-hand side: 
	\begin{eqnarray*}
		\lambda^3\int_{Q} ab^2 \psi_x (a\psi_x^2)_x\sigma^3w^2 = \lambda^3 \left[ \int_0^T \int_0^{\alpha'} ab^2 \psi_x (a\psi_x^2)_x\sigma^3w^2 +  \int_0^T\int_{\omega'} ab^2 \psi_x (a\psi_x^2)_x\sigma^3w^2 +  \int_0^T \int_{\beta'}^1 ab^2 \psi_x (a\psi_x^2)_x\sigma^3w^2  \right]
	\end{eqnarray*}
	
	By the definition of $\psi(x)$ we have
	$$a\psi_x (a\psi_x^2)_x=\begin{cases}
		\frac{x^2}{a^2}(2a-xa'), \ \ & \ \ x\in(0,\alpha')\\
		-\frac{x^2}{a^2}(2a-xa'), \ \ & \ \ x\in(\beta',1)
	\end{cases}$$
	\begin{eqnarray*}
		\hspace*{-0.5cm}\lambda^3\int_{Q} ab^2 \psi_x (a\psi_x^2)_x\sigma^3w^2=\lambda^3 \left[\int_0^T \int_0^{\alpha'} b^2 \frac{x^2}{a^2}(2a-xa')\sigma^3w^2 +  \int_0^T\int_{\omega'} ab^2 \psi_x (a\psi_x^2)_x\sigma^3w^2 - \int_0^T \int_{\beta'}^1 b^2 \frac{x^2}{a^2}(2a-xa')\sigma^3w^2 \right]
	\end{eqnarray*}
	
    By the condition $xa'\geq Ka$ we have $(2a-xa')\geq 2a-Ka\geq a$, and that $\psi$ is bounded in $\omega'$,
	\begin{eqnarray*}
		\hspace*{-0.5cm}\lambda^3\int_{Q} ab^2 \psi_x (a\psi_x^2)_x\sigma^3w^2 \geq \lambda^3 \left[\int_0^T \int_0^{\alpha'} b^2 \frac{x^2}{a}\sigma^3w^2 +  \int_0^T\int_{\omega'} b^2 \sigma^3w^2 -  \int_0^T \int_{\beta'}^1 b^2 \frac{x^2}{a}\sigma^3w^2\right],
	\end{eqnarray*}
    which implies that
	\begin{equation*}
		\int_{Q} ab^2\varphi_x (a\varphi_x^2)_x  w^2\geq\lambda^3 \left[\int_0^T \int_0^{\alpha'} b^2 \frac{x^2}{a}\sigma^3w^2 +  \int_0^T\int_{\omega'} b^2 \sigma^3w^2 -  \int_0^T \int_{\beta'}^1 b^2 \frac{x^2}{a}\sigma^3w^2\right]+2\lambda^4\int_{Q} a^2b^2\sigma^3\psi_x^4  w^2.
	\end{equation*}
	Choosing an appropriate constant in the last inequality, one has: 
	$$s^3\int_{Q} ab^2\varphi_x (a\varphi_x^2)_x  w^2\geq
	C s^3 \lambda^3 \int_0^T \int_0^{\alpha'} b^2 \frac{x^2}{a} \sigma^3 w^2 - Cs^3\lambda^3 \int_0^T\int_{\omega'} b^2 \sigma^3 w^2 + C s^3 \lambda^4 \int_Q b^2 a^2 |\psi_x|^4 \sigma^3 w^2.$$
\end{proof}

The following lemmas aim to estimate the integral:
\begin{equation}
\label{eq:int_I4}
I_4=s\int_{Q}ab^2(a\varphi_x)_{xx} w_x w+2s\int_{Q}ab^2(a\varphi_x )_xw_x^2-s\int_{Q}ab^2 a_x \varphi_xw_{x}^2-s\int_{0}^{T}\left[a^2b^2\varphi_x w_x^2\right]^{x=1}_{x=0}.
\end{equation}

\begin{lema}[For $I_4$ in \eqref{eq:int_I4}]\label{lemaA5}
	\begin{equation*}
		-s \int_{0}^{T}\left[2a^2b^2\varphi_x w_x^2\right]^{x=1}_{x=0} \geq 0.
	\end{equation*}
\end{lema}
\begin{proof}
	Using that $\varphi_x=\lambda \psi_x \sigma$ we have:
	\begin{equation*}
    \begin{split}
		-s \int_{0}^{T}\left[2a^2b^2\varphi_x w_x^2\right]^{x=1}_{x=0} &= -2s\lambda \left[ \int_{0}^{T}a(1)^2b(t)^2 \psi_x(1) \sigma(1,t) w_x^2(1,t) \right]
		+2s\lambda \left[ \int_{0}^{T}a(0)^2b(t)^2 \psi_x(0) \sigma(0,t) w_x^2(0,t) \right] \\
        &=-s\lambda\int_{0}^{T}a(1)^2b(t)^2 \left(-\frac{1}{a(1)}\right) \sigma(1,t) w_x^2(1,t) \geq 0. 
    \end{split}    
	\end{equation*}
\end{proof}

\begin{lema}[For $I_4$ in \eqref{eq:int_I4}]\label{lemaA6}
	\begin{equation*}
		2s \int_Q b^2 (a \varphi_x)_x a w_x^2 \geq
		2 s \lambda \int_0^T \int_0^{\alpha'} b^2 a \sigma w_x^2 - C s\lambda \int_0^T\int_{\omega'} b^2 \sigma w_x^2 + C s \lambda^2 \int_Q b^2 a^2 |\psi_x|^2 \sigma w_x^2.
	\end{equation*}
\end{lema}
\begin{proof}
    First, we have
    $$(a \varphi_x)_x=a_x \varphi_x+a\varphi_{xx} = a_x \lambda \psi_x\sigma+a\lambda \psi_{xx}\sigma+a\lambda^2\psi_x^2\sigma=\lambda(a\psi_x)_x\sigma +a\lambda^2\psi_x^2\sigma.$$
    Then,
    \begin{multline*}
        2s \int_Q b^2 (a \varphi_x)_x a w_x^2 =2s\lambda  \int_Q ab^2 (a \psi_x)_x \sigma w_x^2+2s\lambda^2\int_Q a^2b^2  \psi_x^2 \sigma w_x^2	\\
        =2s\lambda\left[\int_0^T \int_0^{\alpha'} ab^2 (a \psi_x)_x \sigma w_x^2 +  \int_0^T\int_{\omega'} ab^2 (a \psi_x)_x \sigma w_x^2 %\right.\\
		%+ \left.
        + \int_0^T \int_{\beta'}^1 ab^2 (a \psi_x)_x \sigma w_x^2\right] + 2s\lambda^2\int_Q a^2b^2  \psi_x^2 \sigma w_x^2
    \end{multline*}
	
    We know that 
    $\psi_x=\frac{x}{a}, \  (a\psi_x)_x=1, \ \ x\in(0,\alpha')$. We also have that:  $a^2\psi_x^2=a^2\frac{x^2}{a^2}=x^2\geq Ca$. Then,
    \begin{multline*}
        2s \int_Q b^2 (a \varphi_x)_x a w_x^2\geq 2s\lambda\left[\int_0^T \int_0^{\alpha'} ab^2 \sigma w_x^2 +  \int_0^T\int_{\omega'} ab^2 (a \psi_x)_x \sigma w_x^2 - \int_0^T \int_{\beta'}^1 ab^2 \sigma w_x^2\right]	%\\
		+ Cs\lambda^2\int_Q ab^2 \sigma w_x^2.
    \end{multline*}
    
    This inequality implies that,
    $$ 2s \int_Q b^2 (a \varphi_x)_x a w_x^2 \geq 2 s \lambda \int_0^T \int_0^{\alpha'} b^2 a \sigma w_x^2 - C s\lambda \int_0^T\int_{\omega'} b^2 \sigma w_x^2 + C s \lambda^2 \int_Q b^2 a^2 |\psi_x|^2 \sigma w_x^2.$$	
\end{proof}

\begin{lema}[Estimate for $I_4$ of \eqref{eq:int_I4}] \label{lemaA7}
	\begin{equation*}
		-s \int_Q b^2 a \varphi_x a_x w_x^2 \geq
		- K s \lambda \int_0^T \int_0^{\alpha'} b^2 a \sigma w_x^2 - c s \lambda \int_0^T\int_{\omega'} b^2 \sigma w_x^2.
	\end{equation*}
\end{lema}	

\begin{proof} We have that 
	$\varphi_{x}=\lambda \psi_x \sigma$ and $\psi_x=\frac{x}{a}$. Then $xa'\leq Ka$, from the definition of $\psi$, one has that $\psi_x=\frac{x}{a}, \ \  x\in (0,\alpha')$ and $\psi_x=-\frac{x}{a}, \ \  x\in (\beta',1).$
	Then,
    \begin{equation}
    \begin{split}
    \left| -s \int_Q b^2 a \varphi_x a_x w_x^2  \right| &\leq s\int_Q b^2 \lambda \sigma |a a_x \psi_x|  w_x^2 \\
    &\leq s\left[\int_0^T \int_0^{\alpha'} b^2 \lambda \sigma|a a_x \psi_x| w_x^2 +  \int_0^T\int_{\omega'} b^2 \lambda \sigma|a a_x \psi_x|  w_x^2 
		+\int_0^T \int_{\beta'}^1 b^2 \lambda \sigma|a a_x \psi_x|  w_x^2\right] \\
        &\leq s\left[\int_0^T \int_0^{\alpha'} b^2 \lambda \sigma|a a_x \psi_x| w_x^2 +  \int_0^T\int_{\omega'} b^2 \lambda \sigma|a a_x \psi_x|  w_x^2 + \int_0^T \int_{\beta'}^1 b^2 \lambda \sigma|a a_x \psi_x|  w_x^2\right]\\
		&\leq s\left[K\int_0^T \int_0^{\alpha'} ab^2 \lambda  w_x^2 +  \int_0^T\int_{\omega'} b^2 \lambda \sigma w_x^2 + \int_0^T \int_{\beta'}^1 b^2 \lambda \sigma|a a_x \psi_x|  w_x^2\right]
    \end{split}
    \end{equation}
    
	Which implies:
	$$
	-s \int_Q b^2 a \varphi_x a_x w_x^2  \geq -s\left[K\int_0^T \int_0^{\alpha'} ab^2 \lambda  w_x^2 +  \int_0^T\int_{\omega'} b^2 \lambda \sigma w_x^2 + \int_0^T \int_{\beta'}^1 b^2 \lambda \sigma|a a_x \psi_x|  w_x^2\right]
	$$
	$$
  \geq -Ks\lambda \int_0^T \int_0^{\alpha'} ab^2   w_x^2 -  s\lambda\int_0^T\int_{\omega'} b^2  \sigma w_x^2
	$$
\end{proof}

\begin{lema}[Estimate for $I_4$ of \eqref{eq:int_I4}]
\label{lemaA8}
	\begin{equation*}
    \begin{split}
		s \int_Q b^2 a (a \varphi_x)_{xx} w_x w \geq
		&- C s^2\lambda^3 \int_0^T \int_0^{\alpha'} b^2 \frac{x^2}{a}  \sigma^3 w^2 - C \lambda \int_0^T \int_0^{\alpha'} b^2 a \sigma w_x^2
		- Cs^2\lambda^3 \int_0^T\int_{\omega'} b^2 \sigma^3 w^2 \\
		&- C\lambda \int_0^T\int_{\omega'} b^2 \sigma w_x^2
		-Cs^2\lambda^4 \int_Q b^2 a^2 |\psi_x|^4 \sigma^3 w^2 - C\lambda^2 \int_Q b^2 a^2|\psi_x|^2\sigma w_x^2.
    \end{split}
	\end{equation*}
\end{lema}
\begin{proof}
    From the definition of $\varphi$, we have
    $$\varphi_x=\lambda \psi_x \sigma, \ \ \ \ \ \ \varphi_{xx}=\lambda \psi_{xx}\sigma +\lambda^2\psi_{xx}\psi_x \sigma, \ \ \ \ \ \ \varphi_{xxx}=\lambda \psi_{xxx}\sigma +3\lambda^2 \psi_{xx}\psi_x \sigma+\lambda^3\psi_x^3\sigma,$$
    %$$(a\varphi_x)_{xx}=a''\varphi_x+2a'\varphi_{xx}+a\varphi_{xxx}, \ \ \ \ $$
    $$ (a\varphi_x)_{xx}=\left[ \lambda (a\psi_x)_{xx}+2\lambda^2 \psi_x (a\psi_x)_x+\lambda^2a\psi_x\psi_{xx}+\lambda^3 a\psi_x^3   \right]\sigma.$$
    
    Thus, 
	\begin{align*}
		s\int_Q b^2 a (a \varphi_x)_{xx} w_x w=&s\lambda\int_Q ab^2 (a\psi_x)_{xx} \sigma w_x w+2s\lambda^2 \int_Q ab^2 \psi_x (a\psi_x)_x \sigma w_x w\\
		&+s\lambda^2\int_Q a^2b^2\psi_x\psi_{xx} \sigma w_x w
		+s\lambda^3\int_Q a^2b^2\psi_x^3 \sigma w_x w.
	\end{align*}
    
	From the definition of $\psi$ in $(0,\alpha')$ and $(\beta',1)$, we have that $\psi_x=\pm\frac{x}{a}$. With this we get $(a \varphi_x)_{xx}=0$  in those domains. Thus,
		
%	\noindent The first integral in the inequality implies that, 
%	\begin{multline*}
%		s\lambda\int_Q ab^2 (a\psi_x)_{xx} \sigma w_x w=s\lambda\left[\int_0^T \int_0^{\alpha'} ab^2 (a\psi_x)_{xx} \sigma w_x w +  \int_0^T\int_{\omega'} ab^2 (a\psi_x)_{xx} \sigma w_x w \right.\\
%		\left.+\int_0^T \int_{\beta'}^1 ab^2 (a\psi_x)_{xx} \sigma w_x w\right]
%	\end{multline*}
	$$s\lambda\int_Q ab^2 (a\psi_x)_{xx} \sigma w_x w=s\lambda\int_0^T\int_{\omega'} ab^2 (a\psi_x)_{xx} \sigma w_x w $$
%	since $(a\psi_x)_{xx}=0$ in  $\omega'$ in $(\beta', 1).$\\	

	Then, returning to the expression we want to estimate, 
	\begin{align*}
		s\int_Q b^2 a (a \varphi_x)_{xx} w_x w=&s\lambda\int_0^T\int_{\omega'} ab^2 (a\psi_x)_{xx} \sigma w_x w +2s\lambda^2 \int_Q ab^2 \psi_x (a\psi_x)_x \sigma w_x w\\
		&+s\lambda^2\int_Q a^2b^2\psi_x\psi_{xx} \sigma w_x w
		+s\lambda^3\int_Q a^2b^2\psi_x^3 \sigma w_x w
		=J_1+J_2+J_3+J_4.
	\end{align*}
    
    We estimate each integral of the last sum. Taking into account that $\sigma\leq C\sigma^2$, we have
    $$\left|J_1 \right|\leq s\lambda\int_0^T\int_{\omega'} b^2 |a(a\psi_x)_{xx}| \sigma |w_x| |w|\leq Cs\lambda\int_0^T\int_{\omega'} b^2  \sigma^2 |w_x| |w|$$
    Then,
    $$
    \left|J_1 \right|\leq Cs\lambda\int_0^T\int_{\omega'} b \sigma^{3/2}|w| b\sigma^{1/2} |w_x| \leq Cs\lambda\int_0^T\int_{\omega'}b^2 \sigma^3|w|^2 + Cs\lambda\int_0^T\int_{\omega'}b^2 \sigma|w_x|^2.
    $$
	
    \noindent {\bf Estimates for $J_2$:}
    \begin{multline*}
    \left|J_2 \right|\leq s\lambda^2\left[\int_0^T \int_0^{\alpha'} ab^2 \psi_x (a\psi_x)_x \sigma |w_x| |w| +  \int_0^T\int_{\omega'} ab^2 \psi_x (a\psi_x)_x \sigma |w_x| |w| + \int_0^T \int_{\beta'}^1 ab^2 \psi_x (a\psi_x)_x \sigma |w_x| |w|\right]
    \end{multline*}
    On the other hand, we get that $\left| a\psi_x (a\psi_x)_x  \right|=\left| a\cdot \frac{x}{a}\left( a\frac{x}{a} \right)_x \right|=|x|$ for $x\in (0,\alpha')$ and similarly for $x\in (\beta',1)$. We also have the following inequalities: $\sigma\leq C\sigma^2$, \ \ $x\leq a^2|\psi_x|^3$ in $x\in (\beta',1)$. Then,
    \begin{equation*}
    \begin{split}
    \left|J_2 \right| &\leq s\lambda^2\left[\int_0^T \int_0^{\alpha'} b^2 x \sigma |w_x| |w| +  \int_0^T\int_{\omega'} ab^2 \psi_x (a\psi_x)_x \sigma |w_x| |w| + \int_0^T \int_{\beta'}^1 b^2 x \sigma |w_x| |w|\right] \\
    &\leq s\lambda^2 C \left[\int_0^T \int_0^{\alpha'} b^2 x \sigma^2 |w_x| |w| +  \int_0^T\int_{\omega'} b^2  \sigma^2 |w_x| |w| + \int_0^T \int_{\beta'}^1 a^2b^2|\psi_x|^3 \sigma^2 |w_x| |w|\right] \\
    &\leq C \int_0^T \int_0^{\alpha'} \left( \frac{x}{\sqrt{a}}b\sigma^{3/2}\lambda^{3/2}s|w|\right)\left( \sqrt{a}b\sigma^{1/2}\lambda^{1/2}|w_x| \right) \\
    &+C\int_0^T\int_{\omega'} \left( b\sigma^{3/2}\lambda^{3/2}s|w|\right)\left( b\sigma^{1/2}\lambda^{1/2}|w_x| \right)
		+C\int_0^T \int_{\beta'}^1\left( b\sigma^{3/2}sa\lambda^2|\psi_x|^2|w|\right)\left( b\sigma^{1/2}a|\psi_x||w_x| \right)
    \end{split}    
    \end{equation*}
    Thus,
    \begin{equation*}
    \begin{split}
        \left|J_2 \right| \leq &Cs^2\lambda^3\int_0^T \int_0^{\alpha'}\frac{x^2}{a}b^2\sigma^3|w|^2+\lambda\int_0^T \int_0^{\alpha'}ab^2\sigma|w_x|^2+Cs^2\lambda^3\int_0^T\int_{\omega'}b^2\sigma^3|w|^2\\
		&+C\lambda \int_0^T\int_{\omega'}b^2\sigma|w_x|^2
		+Cs^2\lambda^4\int_0^T \int_{\beta'}^1a^2b^2\sigma^3|\psi_x|^4|w|^2+C\int_0^T \int_{\beta'}^1a^2b^2|\psi_x|^2|w_x|^2.    
    \end{split}    
    \end{equation*}

    \noindent {\bf Estimates for $J_3$:}
    $$|J_3|\leq s\lambda^2\int_Q a^2b^2|\psi_x\psi_{xx}| \sigma |w_x| |w|$$
    \begin{multline*}
	|J_3|\leq s\lambda^2\int_0^T \int_0^{\alpha'}          a^2b^2|\psi_x\psi_{xx}| \sigma |w_x| |w| +  s\lambda^2\int_0^T\int_{\omega'} a^2b^2|\psi_x\psi_{xx}| \sigma |w_x| |w|
	+ s\lambda^2\int_0^T \int_{\beta'}^1 a^2b^2|\psi_x\psi_{xx}| \sigma |w_x| |w|.
    \end{multline*}
    
    From the definition of $\psi$, one has,
    $\psi_x\psi_{xx}=\pm \frac{x}{a}\cdot \frac{a-xa'}{a^2}=\pm \frac{x}{a^3}(a-xa')$, for $x\in (0,\alpha')$ and for $x\in (\beta',1)$.
    We also know that $xa'-a\leq Ka-a\leq Ca$. Then,
    \begin{multline*}
		|J_3|\leq s\lambda^2\int_0^T \int_0^{\alpha'} a^2b^2\left|  \frac{a-xa'}{a^3} \right| \sigma |w_x| |w| +  s\lambda^2\int_0^T\int_{\omega'} a^2b^2|\psi_x\psi_{xx}| \sigma |w_x| |w|\\
		+ s\lambda^2\int_0^T \int_{\beta'}^1 a^2b^2\left| - \frac{a-xa'}{a^3} \right| \sigma |w_x| |w|
    \end{multline*}
    $$|J_3|\leq Cs\lambda^2\int_0^T \int_0^{\alpha'} b^2x \sigma^2 |w_x| |w| +  Cs\lambda^2\int_0^T\int_{\omega'} b^2 \sigma^2 |w_x| |w| +C s\lambda^2\int_0^T \int_{\beta'}^1 b^2x \sigma^2 |w_x| |w|,$$
    and we can follow the same estimate as for $J_2$. 
%    The procedure is the same as for $J_2$, that is,
%    \begin{multline*}
%		\left|J_3 \right|\leq Cs^2\lambda^3\int_0^T \int_0^{\alpha'}\frac{x^2}{a}b^2\sigma^3|w|^2+\lambda\int_0^T \int_0^{\alpha'}ab^2\sigma|w_x|^2+Cs^2\lambda^3\int_0^T\int_{\omega'}b^2\sigma^3|w|^2+C\lambda \int_0^T\int_{\omega'}b^2\sigma|w_x|^2\\
%		+Cs^2\lambda^4\int_0^T \int_{\beta'}^1a^2b^2\sigma^3|\psi_x|^4|w|^2+C\int_0^T \int_{\beta'}^1a^2b^2|\psi_x|^2|w_x|^2
%    \end{multline*}
 
    \noindent {\bf Estimates for $J_4:$}
    \begin{equation*}
    \begin{split}
    |J_4| &\leq s\lambda^3\int_Q a^2b^2|\psi_x|^3 \sigma |w_x|  |w|\leq \int_{Q} \left( sab\lambda^{2}|\psi_x|^2\sigma^{3/2}|w|\right)\left( ab\lambda|\psi_x|\sigma^{1/2}|w_x| \right) \\
    &\leq s^2\lambda^4\int_Qa^2b^2|\psi_x|^4 \sigma^3 |w|^2+C\lambda^2\int_{Q} a^2b^2|\psi_x|^2 \sigma |w_x|^2
    \end{split}    
    \end{equation*}
\end{proof}

\begin{lema}\label{lemaA9}
	\begin{multline*}
		\hspace*{-0.5cm}s^3 \lambda^3 \int_0^T \int_0^{\alpha'} b^2 \frac{x^2}{a} \sigma^3 w^2 + s\lambda \int_0^T \int_0^{\alpha'} b^2 a \sigma w_x^2 + s^3\lambda^4 \int_0^T \int_0^1 b^2 a^2 |\psi_x|^4 \sigma^3 w^2 \\
		+ s\lambda^2 \int_0^T \int_0^1 b^2 a^2 |\psi_x|^2 \sigma w_x^2 
		\leq C \left(\int_Q e^{2s\varphi} |h|^2 + s^3\lambda^3 \int_0^T \int_{\omega'} \sigma^3 w^2 + s\lambda \int_0^T \int_{\omega'} \sigma w_x^2 \right)
	\end{multline*}
\end{lema}
\begin{proof}
	We know from Lemma \ref{lemaA1} that
	$$\langle L^+ w,L^- w \rangle =I_1+I_2+I_3+I_4,$$
	and from Lemmas \ref{lemaA2} --  \ref{lemaA8}, we  obtain:
    \begin{equation*}
    \begin{split}
		\langle L^+ w,L^- w \rangle \geq &I_1-Cs^2\lambda^2\left[ \int_{0}^{T} \int_{0}^{\alpha'} b^2 \frac{x^2}{a} \sigma^3 w^2 +  \int_{0}^{T} \int_{\omega'}b^2  \sigma^3 w^2+\int_{0}^{T} \int_{\beta'}^{1}a^2b^2 |\psi_x|^4 \sigma^3 w^2 \right]\\
		&+ C s^3 \lambda^3 \int_0^T \int_0^{\alpha'} b^2 \frac{x^2}{a} \sigma^3 w^2 - Cs^3\lambda^3 \int_0^T\int_{\omega'} b^2 \sigma^3 w^2 + C s^3 \lambda^4 \int_Q b^2 a^2 |\psi_x|^4 \sigma^3 w^2\\
        &+2 s \lambda \int_0^T \int_0^{\alpha'} b^2 a \sigma w_x^2 - C s\lambda \int_0^T\int_{\omega'} b^2 \sigma w_x^2 + C s \lambda^2 \int_Q b^2 a^2 |\psi_x|^2 \sigma w_x^2-Ks\lambda \int_0^T \int_0^{\alpha'} ab^2   w_x^2 \\
        &-  s\lambda\int_0^T\int_{\omega'} b^2  \sigma w_x^2 - C s^2\lambda^3 \int_0^T \int_0^{\alpha'} b^2 \frac{x^2}{a}  \sigma^3 w^2 - C \lambda \int_0^T \int_0^{\alpha'} b^2 a \sigma w_x^2 - Cs^2\lambda^3 \int_0^T\int_{\omega'} b^2 \sigma^3 w^2 \\
        &- C\lambda \int_0^T\int_{\omega'} b^2 \sigma w_x^2 -Cs^2\lambda^4 \int_Q b^2 a^2 |\psi_x|^4 \sigma^3 w^2 - C\lambda^2 \int_Q b^2 a^2|\psi_x|^2\sigma w_x^2    
    \end{split}
    \end{equation*}
%	\begin{multline*}
%		\langle L^+ w,L^- w \rangle \geq I_1-Cs^2\lambda^2\left[ \int_{0}^{T} \int_{0}^{\alpha'} b^2 \frac{x^2}{a} \sigma^3 w^2 +  \int_{0}^{T} \int_{\omega'}b^2  \sigma^3 w^2+\int_{0}^{T} \int_{\beta'}^{1}a^2b^2 |\psi_x|^4 \sigma^3 w^2 \right]+\\
%		C s^3 \lambda^3 \int_0^T \int_0^{\alpha'} b^2 \frac{x^2}{a} \sigma^3 w^2 - Cs^3\lambda^3 \int_0^T\int_{\omega'} b^2 \sigma^3 w^2 + C s^3 \lambda^4 \int_Q b^2 a^2 |\psi_x|^4 \sigma^3 w^2+2 s \lambda \int_0^T \int_0^{\alpha'} b^2 a \sigma w_x^2\\
%		- C s\lambda \int_0^T\int_{\omega'} b^2 \sigma w_x^2 + C s \lambda^2 \int_Q b^2 a^2 |\psi_x|^2 \sigma w_x^2-Ks\lambda \int_0^T \int_0^{\alpha'} ab^2   w_x^2 -  s\lambda\int_0^T\int_{\omega'} b^2  \sigma w_x^2\\
%		- C s^2\lambda^3 \int_0^T \int_0^{\alpha'} b^2 \frac{x^2}{a}  \sigma^3 w^2 - C \lambda \int_0^T \int_0^{\alpha'} b^2 a \sigma w_x^2 - Cs^2\lambda^3 \int_0^T\int_{\omega'} b^2 \sigma^3 w^2 - C\lambda \int_0^T\int_{\omega'} b^2 \sigma w_x^2 \\
%		-Cs^2\lambda^4 \int_Q b^2 a^2 |\psi_x|^4 \sigma^3 w^2 - C\lambda^2 \int_Q b^2 a^2|\psi_x|^2\sigma w_x^2
%	\end{multline*}
Thus, 	
    \begin{multline*}
		\langle L^+ w,L^- w \rangle \geq C \left[ s^3 \lambda^3 \int_0^T \int_0^{\alpha'} b^2 \frac{x^2}{a} \sigma^3 w^2 +s \lambda \int_0^T \int_0^{\alpha'} b^2 a \sigma w_x^2+s^3 \lambda^4 \int_Q b^2 a^2|\psi_x|^4 \sigma^3 w^2 \right.\\
		+s \lambda^2 \int_Q b^2 a^2 |\psi_x|^2 \sigma w_x^2
		\left. - s^3\lambda^3 \int_0^T\int_{\omega'} b^2 \sigma^3 w^2-s\lambda\int_0^T\int_{\omega'} b^2  \sigma w_x^2\right]
    \end{multline*}
    Since
    $$\|L^+ w\|^2+\|L^- w\|^2+2\langle L^+ w,L^- w \rangle=\|e^{s\varphi} h\|^2$$
	Then,
	\begin{multline*}
		\|e^{s\varphi} h\|^2-\|L^+ w\|^2-\|L^- w\|^2 \geq C \left[ s^3 \lambda^3 \int_0^T \int_0^{\alpha'} b^2 \frac{x^2}{a} \sigma^3 w^2 +s \lambda \int_0^T \int_0^{\alpha'} b^2 a \sigma w_x^2 +s^3 \lambda^4 \int_Q b^2 a^2|\psi_x|^4 \sigma^3 w^2 \right.\\
		\left. \hspace*{6cm} +s \lambda^2 \int_Q b^2 a^2 |\psi_x|^2 \sigma w_x^2 
		- s^3\lambda^3 \int_0^T\int_{\omega'} b^2 \sigma^3 w^2-s\lambda\int_0^T\int_{\omega'} b^2  \sigma w_x^2\right].
	\end{multline*}
    
	Consequently
	\begin{multline*}
		\|e^{s\varphi} h\|^2 \geq C \left[ s^3 \lambda^3 \int_0^T \int_0^{\alpha'} b^2 \frac{x^2}{a} \sigma^3 w^2 +s \lambda \int_0^T \int_0^{\alpha'} b^2 a \sigma w_x^2+s^3 \lambda^4 \int_Q b^2 a^2|\psi_x|^4 \sigma^3 w^2\right. \\
		\left.+s \lambda^2 \int_Q b^2 a^2 |\psi_x|^2 \sigma w_x^2
		- s^3\lambda^3 \int_0^T\int_{\omega'} b^2 \sigma^3 w^2-s\lambda\int_0^T\int_{\omega'} b^2  \sigma w_x^2\right].
	\end{multline*}
	This implies the conclusion.
%	\begin{multline*}
%		s^3 \lambda^3 \int_0^T \int_0^{\alpha'} b^2 \frac{x^2}{a} \sigma^3 w^2 +s \lambda \int_0^T \int_0^{\alpha'} b^2 a \sigma w_x^2+s^3 \lambda^4 \int_Q b^2 a^2|\psi_x|^4 \sigma^3 w^2+s \lambda^2 \int_Q b^2 a^2 |\psi_x|^2 \sigma w_x^2 \\
%		\leq C\left[ \int_Q e^{2s\varphi} |h|^2 + s^3\lambda^3 \int_0^T \int_{\omega'} b^2\sigma^3 w^2 + s\lambda \int_0^T \int_{\omega'} b^2\sigma w_x^2  \right]
%	\end{multline*}
\end{proof}

Now we transform the estimate to the original variable $v$.
\begin{lema}\label{lemaA10}
	\begin{multline*}
		s^3 \lambda^3 \int_0^T \int_0^{\alpha'} e^{2s\varphi} \frac{x^2}{a} b^2\sigma^3 v^2 +s \lambda \int_0^T \int_0^{\alpha'} e^{2s\varphi} b^2 a \sigma v_x^2+s^3 \lambda^4 \int_Q e^{2s\varphi}b^2 a^2|\psi_x|^4 \sigma^3 v^2\\
		+s \lambda^2 \int_Q e^{2s\varphi}b^2 a^2 |\psi_x|^2 \sigma v_x^2 
		\leq C\left[ \int_Q e^{2s\varphi} |h|^2 + s^3\lambda^3 \int_0^T \int_{\omega} e^{2s\varphi}\sigma^3 v^2  \right]
	\end{multline*}
\end{lema}
\begin{proof}
	We know that  $v=e^{-s\varphi}w $. Therefore, we have
	$e^{s\varphi}v_x=-s\lambda\psi_x\sigma w+w_x$. 
	Then, it follows:
	$$s \lambda^2 e^{2s\varphi}b^2 a^2 |\psi_x|^2 \sigma v_x^2=\left( s \lambda^2 b^2 a^2|\psi_x|^2\sigma\right)e^{2s\varphi}v_x^2\leq C\left( s \lambda^2 b^2 a^2|\psi_x|^2\sigma\right)\left(  s^2\lambda^2|\psi_x|^2\sigma^2 w^2+w_x^2 \right),$$
	and
	$$s \lambda^2 e^{2s\varphi}b^2 a^2 |\psi_x|^2 \sigma v_x^2\leq  C\left( s^3 \lambda^4 b^2 a^2|\psi_x|^4\sigma^3 w^2 + s \lambda^2 b^2 a^2|\psi_x|^2\sigma w_x^2\right), $$
	Then, since
	$$w_x=s\lambda\psi_x\sigma e^{s\varphi}v+e^{s\varphi}v_x, \qquad w_x^2\leq C \left(  s^2\lambda^2|\psi_x|^2\sigma^2 e^{2s\varphi}v^2+ e^{2s\varphi}v_x^2\right)$$
	and using the minimum value of the function $b(t)$, as $m=\min\limits_{t\in [0,T]} b(t)$, then:
	$$w_x^2\leq \frac{C}{m^2} \left(  s^2\lambda^2b^2|\psi_x|^2\sigma^2 e^{2s\varphi}v^2+ e^{2s\varphi}b^2v_x^2\right),$$
	$$w_x^2\leq C \left(  s^2\lambda^2b^2\sigma^2 e^{2s\varphi}v^2+ e^{2s\varphi}b^2av_x^2\right), \ \ \ x\in w'.$$
%	and
%	\begin{multline*}
%		\hspace*{-1.1cm}\widehat{C}\left(s^3 \lambda^3 \int_0^T \int_0^{\alpha'} e^{2s\varphi}b^2 \frac{x^2}{a} \sigma^3 v^2 +s \lambda \int_0^T \int_0^{\alpha'} b^2 a \sigma w_x^2\right.\\
%		\left.+s^3 \lambda^4 \int_Q b^2 a^2|\psi_x|^4 \sigma^3 w^2
        %+s^3 \lambda^4 \int_Q b^2 a^2|\psi_x|^4 \sigma^3 w^2
%        +s \lambda^2 \int_Q b^2 a^2 |\psi_x|^2 \sigma w_x^2\right) \\
%		\leq C\left[ \int_Q e^{2s\varphi} |h|^2 + s^3\lambda^3 \int_0^T \int_{\omega'} b^2\sigma^3 w^2 + s\lambda \int_0^T \int_{\omega'} b^2\sigma w_x^2  \right].
%	\end{multline*}	 	 
	From the inequalities bellow,
%	$$s \lambda^2 e^{2s\varphi}b^2 a^2 |\psi_x|^2 \sigma v_x^2\leq  C\left( s^3 \lambda^4 b^2 a^2|\psi_x|^4\sigma^3 w^2 + s \lambda^2 b^2 a^2|\psi_x|^2\sigma w_x^2\right),$$ 
	$$s \lambda^2\int_Q e^{2s\varphi}b^2 a^2 |\psi_x|^2 \sigma v_x^2\leq  C\left( s^3 \lambda^4 \int_Q b^2 a^2|\psi_x|^4\sigma^3 w^2 + s \lambda^2 \int_Qb^2 a^2|\psi_x|^2\sigma w_x^2\right),$$
	$$s \lambda \int_0^T \int_0^{\alpha'} b^2 a \sigma w_x^2=
	s \lambda \int_0^T \int_0^{\alpha'} b^2 a \sigma (s^2\lambda^2e^{2s\varphi}|\psi_x|^2\sigma^2  v^2+2s\lambda e^{2s\varphi}|\psi_x|\sigma  v v_x+e^{2s\varphi}v_x^2),$$
    we get 
	\begin{multline*}
		J:=s^3 \lambda^3 \int_0^T \int_0^{\alpha'} e^{2s\varphi} \frac{x^2}{a} b^2\sigma^3 v^2 +s \lambda \int_0^T \int_0^{\alpha'} e^{2s\varphi} b^2 a \sigma v_x^2+s^3 \lambda^4 \int_Q e^{2s\varphi}b^2 a^2|\psi_x|^4 \sigma^3 v^2
		+s \lambda^2 \int_Q e^{2s\varphi}b^2 a^2 |\psi_x|^2 \sigma v_x^2\\
		\leq C \left[  s^3 \lambda^3 \int_0^T \int_0^{\alpha'} b^2 \frac{x^2}{a} \sigma^3 w^2 +s \lambda \int_0^T \int_0^{\alpha'} b^2 a \sigma w_x^2+s^3 \lambda^4 \int_Q b^2 a^2|\psi_x|^4 \sigma^3 w^2+s \lambda^2 \int_Q b^2 a^2 |\psi_x|^2 \sigma w_x^2 \right] \\
%		\leq C\left[ \int_Q e^{2s\varphi} |h|^2 + s^3\lambda^3 \int_0^T \int_{\omega'} b^2\sigma^3 w^2 + s\lambda \int_0^T \int_{\omega'} b^2\sigma w_x^2  \right]
	\end{multline*}
    Thus, from Lemma \ref{lemaA9}, we obtain
%	\begin{align*}
%		s^3 \lambda^3 \int_0^T \int_0^{\alpha'} e^{2s\varphi} \frac{x^2}{a} b^2\sigma^3 v^2 +s \lambda \int_0^T \int_0^{\alpha'} e^{2s\varphi} b^2 a \sigma v_x^2+s^3 \lambda^4 \int_Q e^{2s\varphi}b^2 a^2|\psi_x|^4 \sigma^3 v^2\\
%		+s \lambda^2 \int_Q e^{2s\varphi}b^2 a^2 |\psi_x|^2 \sigma v_x^2
%		\leq C\left[ \int_Q e^{2s\varphi} |h|^2 + s^3\lambda^3 \int_0^T \int_{\omega'} b^2\sigma^3 w^2 + s\lambda \int_0^T \int_{\omega'} b^2\sigma w_x^2  \right]
%	\end{align*}
	\begin{equation}\label{ecA4}
        \begin{split}
%		s^3 \lambda^3 \int_0^T \int_0^{\alpha'} e^{2s\varphi} \frac{x^2}{a} b^2\sigma^3 v^2 +s \lambda \int_0^T \int_0^{\alpha'} e^{2s\varphi} b^2 a \sigma v_x^2+s^3 \lambda^4 \int_Q e^{2s\varphi}b^2 a^2|\psi_x|^4 \sigma^3 v^2 +s \lambda^2 \int_Q e^{2s\varphi}b^2 a^2 |\psi_x|^2 \sigma v_x^2\\
		J &\leq C\left[ \int_Q e^{2s\varphi} |h|^2 + s^3\lambda^3 \int_0^T \int_{\omega'} b^2\sigma^3 w^2 + s\lambda \int_0^T \int_{\omega'} b^2\sigma w_x^2  \right] \\
        &\leq C\left[ \int_Q e^{2s\varphi} |h|^2 + s^3\lambda^3 \int_0^T \int_{\omega'} b^2\sigma^3 w^2 + s\lambda \int_0^T \int_{\omega'} b^2\sigma \left(  s^2\lambda^2b^2\sigma^2 e^{2s\varphi}v^2+ e^{2s\varphi}b^2av_x^2\right)  \right]\\
		&\leq C\left[ \int_Q e^{2s\varphi} |h|^2 + s^3\lambda^3 \int_0^T \int_{\omega'} e^{2s\varphi}b^2\sigma^3 v^2 + s\lambda \int_0^T \int_{\omega'} ab^2\sigma v_x^2  \right]
        \end{split}
	\end{equation}
%	\begin{multline}\label{ecA4}
%		s^3 \lambda^3 \int_0^T \int_0^{\alpha'} e^{2s\varphi} \frac{x^2}{a} b^2\sigma^3 v^2 +s \lambda \int_0^T \int_0^{\alpha'} e^{2s\varphi} b^2 a \sigma v_x^2+s^3 \lambda^4 \int_Q e^{2s\varphi}b^2 a^2|\psi_x|^4 \sigma^3 v^2
%		+ s \lambda^2 \int_Q e^{2s\varphi}b^2 a^2 |\psi_x|^2 \sigma v_x^2\\
%		\leq C\left[ \int_Q e^{2s\varphi} |h|^2 + s^3\lambda^3 \int_0^T \int_{\omega'} e^{2s\varphi}b^2\sigma^3 v^2 + s\lambda \int_0^T \int_{\omega'} ab^2\sigma v_x^2  \right]
%	\end{multline}
    To get the result we estimate the last integral in the right-hand side of (\ref{ecA4}). First, let us take $\chi\in C_0^\infty(\omega)$ such that $0\leq \chi \leq 1$ and $\chi =1 $ in $\omega'$. Multiplying the equation in (\ref{adjunto2}) by $\lambda s e^{2s\varphi}\sigma v \chi$ and integrating over $(0,T)\times (0,1)$, we obtain
%	$$\lambda s\int_0^T \int_0^{1}  e^{2s\varphi}\sigma v v_t \chi +\lambda s \int_0^T \int_0^{1}e^{2s\varphi}\sigma b(t)\left(a(x)v_x\right)_x v\chi =\lambda s \int_0^T \int_0^{1}e^{2s\varphi} \sigma h v \chi $$
	\begin{equation}\label{ecA5}
        \lambda s\int_0^T \int_0^{1}  e^{2s\varphi}\sigma v v_t \chi +\lambda s \int_0^T \int_0^{1}e^{2s\varphi}\sigma b(t)\left(a(x)v_x\right)_x v\chi =\lambda s \int_0^T \int_0^{1}e^{2s\varphi} \sigma h v \chi.
        %\lambda s \int_0^T \int_0^{1}e^{2s\varphi}\sigma  b(t)\left(a(x)v_x\right)_x v\chi =\lambda s \int_0^T \int_0^{1}e^{2s\varphi} \sigma h v \chi-\lambda s\int_0^T \int_0^{1}  e^{2s\varphi}\sigma v v_t \chi
	\end{equation}
    Let us estimate the term %integral of right hand side of (\ref{ecA5})
    \begin{multline*}
        \left|  \int_0^T \int_0^{1}  e^{2s\varphi}\sigma v v_t \chi  \right|=\left| \frac{1}{2} \int_0^T \int_0^{1}  e^{2s\varphi}\sigma \frac{d (v^2)}{dt} \chi \right|=\left| -\frac{1}{2} \int_0^T \int_0^{1}  (e^{2s\varphi}\sigma\chi)_t v^2  \right|
        =\left| \frac{1}{2} \int_0^T \int_0^{1} e^{2s\varphi}\chi (2s\varphi_t\sigma+\sigma_t) v^2  \right|.
    \end{multline*}
    Then,
    \begin{equation}\label{ecA6}
        \left|  \int_0^T \int_0^{1}  e^{2s\varphi}\sigma v v_t \chi \right|\leq C s\int_0^T \int_\omega e^{2s\varphi}\sigma^3 v^2.
    \end{equation}
    On the other hand, % (\ref{ecA5})
    $$ \int_0^T \int_0^{1}e^{2s\varphi}\sigma b\left(a(x)v_x\right)_x v\chi=  \int_0^T \left[ e^{2s\varphi}\sigma b a(x)v_x v\chi \right]_{x=0}^{x=1}- \int_0^T \int_0^{1}\left(e^{2s\varphi}\sigma b v\chi \right)_x a v_x, $$
    %$$ \int_0^T \int_0^{1}e^{2s\varphi}\sigma b\left(a(x)v_x\right)_x v\chi=-  \int_0^T \int_0^{1}e^{2s\varphi} ab \sigma v_x^2\chi -  \int_0^T \int_0^{1} b\left(e^{2s\varphi}\sigma \chi \right)_x a v_x v,$$
    $$\int_0^T \int_0^{1}e^{2s\varphi} ab \sigma v_x^2\chi=-\int_0^T \int_0^{1}e^{2s\varphi}\sigma b\left(a(x)v_x\right)_x v\chi-\int_0^T \int_0^{1} b\left(e^{2s\varphi}\sigma \chi \right)_x a v_x v.$$
    Since $\varphi_x\leq C \sigma$ and $\sigma_x\leq C\sigma$ in $(0,T)\times \omega$, then
    \begin{equation}\label{ecA7}
		\left|  \int_0^T \int_0^{1} b\left(e^{2s\varphi}\sigma \chi \right)_x a v_x v \right|\leq C \int_0^T \int_0^{1} e^{2s\varphi} \sigma^2 |b(t)| |av_x| |v|\leq MC\int_0^T \int_\omega e^{2s\varphi} \sigma^2 |av_x| |v|
    \end{equation}

    Thus,
%    Using the inequalities (\ref{ecA5}) -- (\ref{ecA7}), we obtain
    $$s\lambda \int_0^T \int_\omega e^{2s\varphi} ab \sigma v_x^2 = s\lambda \int_0^T \int_0^{1} e^{2s\varphi} ab \sigma v_x^2\chi,$$
    and using the inequalities (\ref{ecA5}) -- (\ref{ecA7}), that $\lambda\leq \lambda^3$, $s^2\leq s^3$ and $m\leq b(t)$, and proceeding as in Lemma A.11 of \cite{DemarqueLimacoViana_deg_eq2018}, we have 
    \begin{align*}
        &s\lambda \int_0^T \int_\omega e^{2s\varphi} ab \sigma v_x^2	\leq  \left|  -s\lambda\int_0^T \int_0^{1}e^{2s\varphi}\sigma b\left(a(x)v_x\right)_x v\chi-s\lambda\int_0^T \int_0^{1} b\left(e^{2s\varphi}\sigma \chi \right)_x a v_x v  \right|\\
        &\hspace*{-0cm}\leq  s\lambda\int_0^T \int_0^{1} e^{2s\varphi}\sigma |v ||v_t| \chi +s \lambda  \int_0^T \int_0^{1}e^{2s\varphi} \sigma |h| |v| \chi  +s\lambda\int_0^T \int_0^{1} b\left(e^{2s\varphi}\sigma \chi \right)_x |a v_x|  |v|\\
        &\hspace*{-0cm}\leq Cs^2\lambda\int_0^T \int_\omega e^{2s\varphi}\sigma^3 v^2+s \lambda  \int_0^T \int_\omega  |e^{s\varphi} h| |e^{s\varphi}\sigma v| +Cs\lambda\int_0^T \int_\omega e^{2s\varphi} \sigma^2 |av_x| |v|\\
        &\leq Cs^2\lambda\int_0^T \int_\omega e^{2s\varphi}\sigma^3 v^2+\frac{1}{2}s \lambda  \int_0^T \int_\omega  e^{2s\varphi} h^2 +\frac{1}{2}s \lambda  \int_0^T \int_\omega  e^{2s\varphi}\sigma^2 v^2
        +Cs \lambda \int_0^T \int_\omega\left( e^{s\varphi}\sigma^{1/2}\sqrt{a}|v_x| \right)\left( e^{s\varphi}\sqrt{a}\sigma^{3/2}|v|  \right)\\
        &\leq Cs^2\lambda\int_0^T \int_\omega e^{2s\varphi}\sigma^3 v^2+C  \int_0^T \int_\omega   e^{2s\varphi} h^2 +\varepsilon Cs\lambda \int_0^T \int_\omega e^{2s\varphi} \sigma a v_x^2 + C_\varepsilon s\lambda\int_0^T \int_\omega e^{2s\varphi} a\sigma^3v^2 \\
        &\leq C\left( \int_0^T \int_\omega   e^{2s\varphi} h^2+  s^3\lambda^3\int_0^T \int_\omega e^{2s\varphi}\sigma^3 v^2 \right).    
    \end{align*}  
%    Let us take as $\varepsilon =\frac{1}{2}C.$ We know that $\lambda\leq \lambda^3$, $s^2\leq s^3$ and $m\leq b(t)$. Then,
%    \begin{equation*}
%    \begin{split}
%    		s\lambda \int_0^T \int_\omega e^{2s\varphi} ab \sigma v_x^2 &\leq Cs^3\lambda^3\int_0^T \int_\omega e^{2s\varphi}\sigma^3 v^2+C  \int_0^T \int_\omega   e^{2s\varphi} h^2 +\frac{1}{2}s\lambda \int_0^T \int_\omega e^{2s\varphi} \sigma a b v_x^2
%		+C_\varepsilon s\lambda\int_0^T \int_\omega e^{2s\varphi} a\sigma^3v^2\\
%        &\leq C\left( \int_0^T \int_\omega   e^{2s\varphi} h^2+  s^3\lambda^3\int_0^T \int_\omega e^{2s\varphi}\sigma^3 v^2 \right).        
%    \end{split}    
%    \end{equation*}
	This last estimate, together with estimate (\ref{ecA4}), completes the proof.	
\end{proof}

\noindent Now we are ready to prove Proposition \ref{prop_carleman_11}.

\begin{proof}[\bf Proof of Proposition \ref{prop_carleman_11}]
    Notice that
    \begin{align*}
        s^2\lambda^2\int_Qe^{2s\varphi}b^2\sigma^2v^2 &= s^2\lambda^2\int_Qb^2\sigma^2w^2 = \int_{Q}\left(\lambda^{3/2}s^{3/2}b\sigma^{3/2}\frac{x}{\sqrt{a}}w\right)\left( \lambda^{1/2}s^{1/2}b\sigma^{1/2}\frac{\sqrt{a}}{x}w \right) \\
        &\leq \frac{1}{2}s^3\lambda^3\int_{Q}b^2\sigma^3\frac{x^2}{a}w^2+\frac{1}{2}s\lambda\int_Qb^2\sigma \frac{a}{x^2}w^2=\mathcal{I}_1+\mathcal{I}_2.
    \end{align*}
   %$$s^2\lambda^2\int_Qe^{2s\varphi}b^2\sigma^2v^2=s^2\lambda^2\int_Qb^2\sigma^2w^2=\int_{Q}\left(\lambda^{3/2}s^{3/2}b\sigma^{3/2}\frac{x}{\sqrt{a}}w\right)\left( \lambda^{1/2}s^{1/2}b\sigma^{1/2}\frac{\sqrt{a}}{x}w \right)$$

    %$$\leq \frac{1}{2}s^3\lambda^3\int_{Q}b^2\sigma^3\frac{x^2}{a}w^2+\frac{1}{2}s\lambda\int_Qb^2\sigma \frac{a}{x^2}w^2=\mathcal{I}_1+\mathcal{I}_2,$$
    On the one hand, since $x^2/a$ is bounded in $\omega'$ and $x^2/a\leq a^2|\psi_x|^2$ in $[\beta',1]$, we have
    \begin{align*}
        \mathcal{I}_1= s^3\lambda^3 \left( \int_0^T \int_0^{\alpha'} b^2\sigma^3\frac{x^2}{a}w^2 +  \int_0^T\int_{\omega'} b^2\sigma^3\frac{x^2}{a}w^2 + \int_0^T \int_{\beta'}^1b^2\sigma^3\frac{x^2}{a}w^2 \right) \\
        \leq s^3\lambda^3 \left( \int_0^T \int_0^{\alpha'} b^2\sigma^3\frac{x^2}{a}w^2 + C \int_0^T\int_{\omega'} b^2\sigma^3w^2 + \int_0^T \int_{\beta'}^1a^2b^2\sigma^3|\psi_x|^4w^2 \right).
    \end{align*}
 %   $$\mathcal{I}_1=\int_0^T \int_0^{\alpha'} b^2\sigma^3\frac{x^2}{a}w^2 +  \int_0^T\int_{\omega'} b^2\sigma^3\frac{x^2}{a}w^2 + \int_0^T \int_{\beta'}^1b^2\sigma^3\frac{x^2}{a}w^2.$$

%    Since $x^2/a$ is bounded in $\omega'$ and $x^2/a\leq a^2|\psi_x|^2$ in $[\beta',1]$, we have
%    $$\mathcal{I}_1\leq \int_0^T \int_0^{\alpha'} b^2\sigma^3\frac{x^2}{a}w^2 +  \int_0^T\int_{\omega'} b^2\sigma^3w^2 + \int_0^T \int_{\beta'}^1a^2b^2\sigma^3|\psi_x|^4w^2.$$
    
    Therefore, using Lemma \ref{lemaA9}, one has
    $$\mathcal{I}_1\leq C\left[ \int_Q e^{2s\varphi} |h|^2 + s^3\lambda^3 \int_0^T \int_{\omega'} b^2\sigma^3 w^2 + s\lambda \int_0^T \int_{\omega'} b^2\sigma w_x^2  \right].$$
    
    {\bf Estimate for $\mathcal{I}_2$}
    We proceed as in the proof of Proposition A.1 of \cite{DemarqueLimacoViana_deg_eq2018}, using Hardy-Poincaré inequality,
    \begin{eqnarray*}
            \mathcal{I}_2&=&\frac{1}{2}s\lambda\int_Qb^2\sigma \frac{a}{x^2}w^2=\frac{1}{2}s\lambda\int_Q \frac{a}{x^2}(b\sigma^{1/2}w)^2\leq C s\lambda\int_Q a(b\sigma^{1/2}w)_x^2\\
            &\leq &C s\lambda\int_Q ab^2\left( \frac{1}{2}\sigma^{-1/2}\sigma_x w+\sigma^{1/2}w_x  \right)^2\leq C s\lambda\int_Q ab^2\left( \frac{1}{4}\sigma^{-1}\sigma_x^2 w^2+\sigma w_x^2  \right).
    \end{eqnarray*}
    
    Thus
    \begin{eqnarray*}
            \mathcal{I}_2&\leq &C s\lambda\int_Q ab^2\sigma^{-1}\sigma_x^2 w^2+C s\lambda\int_Q ab^2\sigma w_x^2\\
            &\leq &C s\lambda^3\int_Q ab^2 \sigma |\psi_x|^2 w^2+C s\lambda\int_Q ab^2\sigma w_x^2\\
            &\leq &C s\lambda^3 \left[\int_0^T \int_0^{\alpha'} ab^2 \sigma |\psi_x|^2 w^2 +  \int_0^T\int_{\omega'} ab^2 \sigma |\psi_x|^2 w^2 + \int_0^T \int_{\beta'}^1ab^2 \sigma |\psi_x|^2 w^2\right]\\
            && +C s\lambda \left[\int_0^T \int_0^{\alpha'} ab^2\sigma w_x^2 +  \int_0^T\int_{\omega'} ab^2\sigma w_x^2 + \int_0^T \int_{\beta'}^1ab^2\sigma w_x^2 \right]\\
            &\leq &C s\lambda^3 \left[\int_0^T \int_0^{\alpha'} b^2 \frac{x^2}{a}\sigma^3  w^2 +  \int_0^T\int_{\omega'} b^2 \sigma^3 w^2 +\lambda \int_0^T \int_{\beta'}^1a^2b^2 \sigma |\psi_x|^4 w^3\right]\\
            && +C s\lambda \left[\int_0^T \int_0^{\alpha'} ab^2\sigma w_x^2 +  \int_0^T\int_{\omega'} b^2\sigma w_x^2 + \lambda\int_0^T \int_{\beta'}^1a^2b^2|\psi_x|^2\sigma w_x^2 \right]\\
            &\leq& C\left[ \int_Q e^{2s\varphi} |h|^2 + s^3\lambda^3 \int_0^T \int_{\omega'} b^2\sigma^3 w^2 + s\lambda \int_0^T \int_{\omega'} b^2\sigma w_x^2  \right].
	\end{eqnarray*}
    
	Then
	$$s^2\lambda^2\int_Qe^{2s\varphi}b^2\sigma^2v^2\leq \mathcal{I}_1+\mathcal{I}_2\leq C\left[ \int_Q e^{2s\varphi} |h|^2 + s^3\lambda^3 \int_0^T \int_{\omega'} b^2\sigma^3 w^2 + s\lambda \int_0^T \int_{\omega'} b^2\sigma w_x^2  \right]$$
	using the same technique as the Lemma \ref{lemaA10}.
\end{proof}

%---------------- FIN LEMMAS REVISADO --------------------

To finish the proof of Proposition \ref{prop_carleman_1} we show that we can absorve the terms of order zero and one. 

\begin{proof} (of Carleman inequality)
    
If $v$ is a solution of \eqref{adjunto2_texto_art}, then $v$ is also a solution of \eqref{adjunto2} with $h = F + B\sqrt{a} v_x -c v$. In this case, applying Proposition \ref{prop_carleman_11}, there exist $C > 0$, $\lambda_0 > 0$ and $s_0 > 0$ such that $v$ satisfies, for all $s \geq s_0$ and $\lambda\geq \lambda_0$
	\begin{equation}\label{ecultimo}
	\int_{0}^{T}\int_{0}^{1}e^{2s\varphi}\left(  (s\lambda)\sigma a b^2 v_x^2+(s\lambda)^2\sigma^2b^2 v^2 \right)\leq C\left(	\int_{0}^{T}\int_{0}^{1} e^{2s\varphi} |h|^2+(\lambda s)^3\int_{0}^{T}\int_{\omega}e^{2s\varphi} \sigma^3  v^2  \right)
\end{equation}

Recalling that $c\in L^\infty(Q)$, $B\in L^\infty(Q)$, $\sigma \geq C > 0$ and $m\leq b(t)$, we have
\begin{eqnarray*}
	\int_{0}^{T}\int_{0}^{1} e^{2s\varphi} |h|^2&=&\int_{0}^{T}\int_{0}^{1} e^{2s\varphi} |F+B\sqrt{a} v_x -cv|^2\\
	&\leq &C\int_{0}^{T}\int_{0}^{1} e^{2s\varphi}|F|^2+\|B\|_{L^\infty(Q)}C\int_{0}^{T}\int_{0}^{1} e^{2s\varphi}a|v_x|^2+\|c\|_{L^\infty(Q)}C\int_{0}^{T}\int_{0}^{1} e^{2s\varphi}|v|^2\\
		&\leq &C\int_{0}^{T}\int_{0}^{1} e^{2s\varphi}|F|^2+C\int_{0}^{T}\int_{0}^{1} e^{2s\varphi}b^2\sigma a|v_x|^2+C\int_{0}^{T}\int_{0}^{1} e^{2s\varphi}b^2\sigma^2 |v|^2.
\end{eqnarray*}

Therefore, taking $\lambda_0$ and $s_0$ large enough, the last integral can be absorbed by the left-hand side of \eqref{ecultimo}, which complete the proof.

\end{proof}

\section{Carleman estimate for a system of degenerate non-autonomous system}
\label{sec:carleman_system_nonautonomous}

Now we prove a Carleman estimate for the system \eqref{eq:adjoint_PDE} which allows us to obtain the controllability with control acting only on one of the equations.

\begin{propo} %\cite{DemarqueLimacoViana_deg_sys2020}
There exist positive constants $C$, $\lambda_0$ and $s_0$ such that, for any $s \geq s_0$, $\lambda \geq \lambda_0$ and any $\phi^T \in L^2(Q)$, the corresponding solution $(\phi,\psi)$ of \eqref{eq:adjoint_PDE} satisfies
\begin{equation}\label{carleman}
I(\phi) + I(\psi) \leq C \left( \int_Q e^{2s\varphi} s^4 \lambda^4 \sigma^4 (|F_1|^2+|F_2|^2)dxdt + \int_{\omega\times(0,T)} e^{2s\varphi} s^8 \lambda^8 \sigma^8 \phi^2dxdt \right).
\end{equation}
\end{propo}

\begin{proof}
    We follow the method of Proposition 3 of \cite{DemarqueLimacoViana_deg_sys2020}.

In general we consider an adjoint system of the form:
\begin{equation}
\label{eq:adj_system_gen}
\begin{cases}
    -y_t - b(t)(a(x)y_x)_x - d_1(x,t)\sqrt{a}y_x + b_{11}y + b_{21}z = F_1,\qquad & \text{in } \Omega \times (0,T) \\
    -z_t - b(t)(a(x)z_x)_x - d_2(x,t)\sqrt{a}z_x + b_{12}y + b_{22}z = F_2,\qquad & \text{in } \Omega \times (0,T) \\
    y(t,0)=y(0,1)=z(t,0)=z(t,1)=0, \qquad & \text{on } (0,T) \\
    y(T,x)=y_T(x), \quad z(T,x)=z_T(x), \qquad & \text{in } \Omega, 
\end{cases}    
\end{equation}
with bounded coefficients. First, we rewrite the first equation of \eqref{eq:adj_system_gen} as
$$
-y_t - b(t)(a(x)y_x)_x - d_1(x,t)\sqrt{a}y_x + b_{11}y  = F_1 -b_{21}z,
$$
and applying Proposition \ref{prop_carleman_1} we obtain
$$
I(y) \leq C \left( \int_Q e^{2s \varphi}|F_1|^2 + \|b_{21}\|_{L^\infty} \int_Q e^{2s \varphi}|z|^2 + \int_{\omega \times (0,T)} e^{2s \varphi} \lambda^3 s^3 \sigma^3 |y|^2 \right).
$$

For the second equation of \eqref{eq:adj_system_gen} we get, analogously,
$$
I(z) \leq C \left( \int_Q e^{2s \varphi}|F_2|^2 + \|b_{12}\|_{L^\infty} \int_Q e^{2s \varphi}|y|^2 + \int_{\omega \times (0,T)} e^{2s \varphi} \lambda^3 s^3 \sigma^3 |z|^2 \right).
$$

Adding these two inequalities, and taking $s$ and $\lambda$ sufficiently large, so that $I(y)$ will absorb the term with integrand $e^{2s \varphi} |y|^2$ and $I(z)$ will absorb the term with integrand $e^{2s \varphi} |z|^2$, we get
$$
I(y)+I(z) \leq C \left( \int_Q e^{2s \varphi} (|F_1|^2 + |F_2|^2) + \int_{\omega \times (0,T)} e^{2s \varphi} \lambda^3 s^3 \sigma^3 (|y|^2 +|z|^2) \right).
$$

Thus, to get the result it is enough to show that for a small $\epsilon>0$, we have

$$
\int_{\omega \times (0,T)} e^{2s \varphi} \lambda^3 s^3 \sigma^3 |z|^2 \leq \epsilon I(z) + C \left( \int_Q e^{2s \varphi} s^4 \lambda^4 \sigma^4 (|F_1|^2 + |F_2|^2) + \int_{\omega \times (0,T)} e^{2s \varphi} \lambda^8 s^8 \sigma^8 |y|^2 \right).
$$

We take $\chi \in C_0^\infty(\omega)$ such that $0\leq \chi \leq 1$ and $\chi \equiv 1$ in $\omega_1$. Since $\inf b_{21} >0$ on $\omega_1$, we see that
$$
\int_{\omega_1 \times (0,1)} e^{2s\varphi} \lambda^3 s^3 \sigma^3 |z|^2 \leq C \int_{\omega \times (0,1)} \chi b_{21} e^{2s\varphi} \lambda^3 s^3 \sigma^3 |z|^2.
$$

Multiplying the first equation in \eqref{eq:adj_system_gen} by $e^{2s\varphi}\lambda^3 s^3 \sigma^3 \chi z$ and integrating over $Q$ we get
\begin{align*}
%\begin{split}
    \int_Q \chi b_{21} e^{2s\varphi} \lambda^3 s^3 \sigma^3 |z|^2 &= \int_Q \chi e^{2s\varphi} \lambda^3 s^3 \sigma^3 z F_1 + \int_Q \chi e^{2s\varphi} \lambda^3 s^3 \sigma^3 z y_t \\
    & +\int_Q \chi e^{2s\varphi} \lambda^3 s^3 \sigma^3 b(t) (a y_x)_x z + \int_Q \chi d_1 e^{2s\varphi} \lambda^3 s^3 \sigma^3 \sqrt{a} y_x z - \int_Q \chi b_{11} e^{2s\varphi} \lambda^3 s^3 \sigma^3 y z \\
    &= I_1 + I_2 + I_3 + I_4 + I_5.
%\end{split}    
\end{align*}

The integrals for $I_1$, $I_2$ and $I_5$ are the same as in Proposition 3 of \cite{DemarqueLimacoViana_deg_sys2020} and the integral $I_3$ has a $b(t)$ in the integrand  which does not affect the integrations by parts. In our case we get
\begin{equation*}
\begin{split}
I_3 = &\int_Q \lambda^3 s^3 b(t)(a(\chi e^{2s\varphi} \sigma^3)_x)_x yz - \int_Q
b(t) \chi \lambda^3 s^3 \sigma^3 e^{2s\varphi} a z_x y_x + \int_Q b(t) \lambda^3 s^3 (\chi e^{2s\varphi} \sigma^3)_x a z_x y,
\end{split}    
\end{equation*}
and, using the second equation in \eqref{eq:adj_system_gen}, we get 
\begin{equation*}
\begin{split}
I_2= &\int_Q \lambda^3 s^3 \left[ -\chi (e^{2s\varphi})_t - \chi \sigma^3 e^{2s\varphi}b_{22} \right] yz  + \int_Q \chi \lambda^3 s^3 \sigma^3 e^{2s\varphi} \left[ b(t) (a z_x)_x y + d_2 \sqrt{a} z_x y - b_{12}y^2 + y F_2 \right] \\
= & \int_Q \lambda^3 s^3 \left[ -\chi (e^{2s\varphi})_t - \chi \sigma^3 e^{2s\varphi}b_{22} \right] yz - \int_Q b(t) \lambda^3 s^3 (\chi e^{2s\varphi} \sigma^3)_x a z_x y - \int_Q b(t) \chi \lambda^3 s^3 \sigma^3 e^{2s\varphi} a z_x y_x \\
&+ \int_Q \chi \lambda^3 s^3 \sigma^3 e^{2s\varphi} d_2 \sqrt{a} z_x y -\int_Q \chi \lambda^3 s^3 \sigma^3 e^{2s\varphi}b_{12}y^2 + \int_Q \chi \lambda^3 s^3 \sigma^3 e^{2s\varphi} y F_2.
\end{split}    
\end{equation*}

Thus,
\begin{equation}\label{eq:I2_I3}
\begin{split}
I_2 + I_3 &= \int_Q \lambda^3 s^3 \left[ -\chi (e^{2s\varphi})_t - \chi \sigma^3 e^{2s\varphi}b_{22}  + b(t)(a(\chi e^{2s\varphi} \sigma^3)_x)_x \right] yz + \int_Q \chi \lambda^3 s^3 \sigma^3 e^{2s\varphi} d_2 \sqrt{a} z_x y \\
&- 2 \int_Q b(t) \chi \lambda^3 s^3 \sigma^3 e^{2s\varphi} a z_x y_x -\int_Q \chi \lambda^3 s^3 \sigma^3 e^{2s\varphi}b_{12}y^2 + \int_Q \chi \lambda^3 s^3 \sigma^3 e^{2s\varphi} y F_2 \\
&= J_1 + J_2 + J_3 + J_4 + J_5.
\end{split}    
\end{equation}

Using the fact that $m \leq b(t) \leq M$, the estimates of $J_1$, $J_3$, $J_4$ and $J_5$ are the same as in Proposition 3 of \cite{DemarqueLimacoViana_deg_sys2020}. On the oher hand, by Young's inequality
\begin{equation*}
\begin{split}
J_2 &\leq \frac{C \epsilon}{m^2} \int_Q \lambda s \sigma e^{2s\varphi} b^2 a z_x^2 + C_\epsilon \int_Q \chi \lambda^5 s^5 \sigma^5 e^{2s\varphi} y^2 \\
&\leq \epsilon I(z) + C_\epsilon \int_{\omega \times (0,1)} e^{2s\varphi} \lambda^8 s^8 \sigma^8 |y|^2.    
\end{split}
\end{equation*}

Finally, we estimate $I_4$. Using that $d_1$ is bounded and by Young's inequality, we have
$$
I_4 \leq \frac{C \epsilon}{m^2}
\int_Q e^{2s\varphi} \lambda^2  s^2 \sigma^2 b^2 |z|^2 + C_\epsilon \int_Q e^{2s\varphi} \lambda^4  s^4 \sigma^4 a |y_x|^2 \leq \epsilon I(z) + C \int_Q e^{2s\varphi} \lambda^4  s^4 \sigma^4 a |y_x|^2,
$$
and the last term is estimated as in Proposition 3 of \cite{DemarqueLimacoViana_deg_sys2020}, by multiplying the first equation of \eqref{eq:adj_system_gen} by $\chi^2 e^{2s\varphi}\lambda^5 s^5 \sigma^5 y$ and integrating in $Q$. Integration by parts in the term $(a y_x)_x$ allows us to estimate $\int_Q e^{2s\varphi} \lambda^4  s^4 \sigma^4 a |y_x|^2$ by integrals of the form already estimated in \eqref{eq:I2_I3}.  
%\textcolor{red}{Después mando las cuentas de esta estimativa}
%...    
\end{proof}

In order to get the global null controllability of the linearized system, we need a Carleman inequality with weights which do not vanish at $t=0$. Consider the function $m \in C^\infty([0,T])$ satisfying $m(0)>0$,
\begin{equation}
\label{eq:def_m}
m(t) \geq t^4(T-t)^4, \quad t \in (0,T/2], \qquad\qquad m(t) = t^4(T-t)^4, \quad t \in [T/2,T],    
\end{equation}
and take
$$
\tau(t) = \frac{1}{m(t)}, \qquad \zeta(x,t) = \tau(t) \eta(x), \qquad A(x,t)  = \tau(t) (e^{\lambda(|\Psi|_\infty + \Psi)}-e^{3\lambda|\Psi|_\infty}).
$$

We use the following notation:
$$
\Gamma_0(\phi,\psi) = \int_Q e^{2s A}((s\lambda) \zeta b^2 a (|\phi_x|^2 + |\psi_x|^2) + (s\lambda)^2 {\zeta}^2 b^2 (|\phi|^2 + |\psi|^2))dxdt.
$$
%$$
%\Gamma(\phi,\varrho,\tilde \varrho) = \int_Q e^{2s A}((s\lambda) \zeta b^2 a (|\phi_x|^2+|\varrho_x|^2+|\tilde \varrho_x|^2) + (s\lambda)^2 {\zeta}^2 b^2 (|\phi|^2 + |\varrho|^2 + |\tilde \varrho|^2))dxdt.
%$$

\begin{propo} %\cite{DemarqueLimacoViana_deg_sys2020}
\label{prop:carleman}
There exist positive constants $C$, $\lambda_0$ and $s_0$ such that, for any $s \geq s_0$, $\lambda \geq \lambda_0$ and any $\phi^T\in L^2(Q)$, the corresponding solution $(\phi,\psi)$ of \eqref{eq:adjoint_PDE} satisfies
$$
\Gamma_0(\phi,\psi) \leq C \left( \int_Q e^{2s A} s^4 \lambda^4 \zeta^4 (|F_1|^2+|F_2|^2)dxdt + \int_{\omega \times (0,T)} e^{2s A} s^8 \lambda^8 \zeta^8 |\phi|^2 dxdt\right).
$$
\end{propo}
\begin{proof}
%    We follow the computations of Proposition 4 of \cite{DemarqueLimacoViana_deg_sys2020}.
As the previous proposition, we work in a general system of the form \eqref{eq:adj_system_gen} in the variables $(y,z)$. 
We decompose the integral $\Gamma_0(y,z)$ as	
	$$\Gamma_0(y,z)=
	\Gamma^1_0(y,z)+\Gamma^2_0(y,z)$$
and define the following operators:
$$
\Gamma^1_0(y,z) =\int_{0}^{T/2}\int_{0}^{1} e^{2s A}((s\lambda) \zeta b^2 a (|y_x|^2 + |z_x|^2) + (s\lambda)^2 {\zeta}^2 b^2 (|y|^2 + |z|^2))dxdt,
$$
$$
	\Gamma^2_0(y,z) =\int_{T/2}^{T}\int_{0}^{1} e^{2s A}((s\lambda) \zeta b^2 a (|y_x|^2 + |z_x|^2) + (s\lambda)^2 {\zeta}^2 b^2 (|y|^2 + |z|^2))dxdt.
$$

The strategy is to prove that, for $i=1,2$,
$$
\Gamma^i_0(y,z)\leq C \left( \int_Q e^{2sA} s^4 \lambda^4 \zeta^4 (|F_1|^2+|F_2|^2)dxdt + \int_{\omega \times (0,T)} e^{2sA} s^8 \lambda^8 \zeta^8 y^2 dxdt \right).
$$
%	$$\Gamma^2_0(y,z)\leq C \left( \int_Q e^{2sA} s^4 \lambda^4 \zeta^4 (|F_1|^2+|F_2|^2)dxdt + \int_{O\times(0,T)} e^{2sA} s^8 \lambda^8 \zeta^8 \phi^2dxdt \right)
%    $$
    
To estimate $\Gamma^2_0(y,z)$, let us observe that $e^{2s\varphi}\sigma^n\leq Ce^{2sA}\zeta^n$ for all $(x,t)\in [0,1]\times[0,T]$ and $n\geq 0$. Since $\tau=\theta$ and $A=\varphi$ in $[T/2, T]$, Carleman inequality \eqref{carleman} implies	
$$\Gamma^2_0(y,z) \leq C \left( \int_Q e^{2sA} s^4 \lambda^4 \zeta^4 (|F_1|^2+|F_2|^2)dxdt + \int_{O\times(0,T)} e^{2sA} s^8 \lambda^8 \zeta^8 |y|^2dxdt \right).$$
     
Now, we will prove an analogous estimate for $\Gamma^1_0(y,z)$ arguing as in \cite{DemarqueLimacoViana_deg_sys2020}. Multiplying the first and the second equations of \eqref{eq:adj_system_gen} by $y$ and $z$, respectively, and integrating over $[0, 1]$, we obtain	
	
$$-\int_{0}^{1}y_ty - \int_{0}^{1}b(t)(a(x)y_x)_x y - \int_{0}^{1}d_1(x,t)\sqrt{a}y y_x + \int_{0}^{1}b_{11}|y|^2 + \int_{0}^{1}b_{21}|z||y| =\int_{0}^{1} |F_1| |y| $$
$$-\int_{0}^{1}z_tz - \int_{0}^{1}b(t)(a(x)z_x)_x z - \int_{0}^{1}d_2(x,t)\sqrt{a}z z_x + \int_{0}^{1}b_{12}|y||z| + \int_{0}^{1}b_{22}|z|^2 =\int_{0}^{1} |F_2| |z|.$$

Thus, since $b$ and $d_i$, $i=1,2$, are bounded, for sufficiently small $\epsilon>0$, 
\begin{multline}\label{ecsiste3102}
    -\frac{1}{2}\frac{d}{dt}\left(\|y\|^2_{L^2(0,1)}+\|z\|^2_{L^2(0,1)}\right)-C_\epsilon\left(\|y\|^2_{L^2(0,1)}+\|z\|^2_{L^2(0,1)}\right) - \epsilon \left(\|\sqrt{a}y_x\|^2_{L^2(0,1)}+\|\sqrt{a}z_x\|^2_{L^2(0,1)}\right) \\    + m\left(\|\sqrt{a}y_x\|^2_{L^2(0,1)}+\|\sqrt{a}z_x\|^2_{L^2(0,1)}\right) \leq \|F_1\|^2_{L^2(0,1)}+\|F_2\|^2_{L^2(0,1)}, 
\end{multline}
which implies, for $C>0$, 
$$-\frac{d}{dt}\left[e^{Ct}\left(\|y\|^2_{L^2(0,1)}+\|z\|^2_{L^2(0,1)}\right)\right]\leq e^{Ct}\left(\|F_1\|^2_{L^2(0,1)}+\|F_2\|^2_{L^2(0,1)}\right)$$

We can continue as in \cite{DemarqueLimacoViana_deg_sys2020}.
Integrating from a $t\in [0,T/2]$ to $t+T/4$, we get
%$$-\int_{t}^{t+T/4}\frac{d}{dt}\left[e^{Ct}\left(\|y\|^2_{L^2(0,1)}+\|z\|^2_{L^2(0,1)}\right)\right]dt\leq\int_{t}^{t+T/4} e^{Ct}\left(\|F_1\|^2_{L^2(0,1)}+\|F_2\|^2_{L^2(0,1)}\right)dt$$
\begin{multline*}
    e^{Ct}\left(\|y\|^2_{L^2(0,1)}+\|z\|^2_{L^2(0,1)}\right)-e^{C(t+T/4)}\left(\|y(t+T/4)\|^2_{L^2(0,1)}+\|z(t+T/4)\|^2_{L^2(0,1)}\right)\\ \leq\int_{t}^{t+T/4} e^{Ct}\left(\|F_1\|^2_{L^2(0,1)}+\|F_2\|^2_{L^2(0,1)}\right)dt 
\end{multline*}
%\vspace*{-1.2cm}
Thus, 
\begin{multline*}
\|y\|^2_{L^2(0,1)}+\|z\|^2_{L^2(0,1)} \leq e^{CT}\int_{0}^{3T/4} \left(\|F_1\|^2_{L^2(0,1)}+\|F_2\|^2_{L^2(0,1)}\right)dt\\
 +e^{3CT/4}\left(\|y(t+T/4)\|^2_{L^2(0,1)}+\|z(t+T/4)\|^2_{L^2(0,1)}\right)
\end{multline*}

Integrating again from 0 to $T/2$, we have	
\begin{equation}\label{ecsis3112}
\begin{split}
\int_{0}^{T/2}\left(\|y\|^2_{L^2(0,1)}+\|z\|^2_{L^2(0,1)}\right) %&\leq 
%\int_{0}^{T/2}\left[e^{CT}\int_{0}^{3T/4} \left(\|F_1(t)\|^2_{L^2(0,1)}+\|F_2(t)\|^2_{L^2(0,1)}\right)dt\right]dt\\
%&+\int_{0}^{T/2}e^{3CT/4}\left(\|y(t+T/4)\|^2_{L^2(0,1)}+\|z(t+T/4)\|^2_{L^2(0,1)}\right)\\
%\end{multline*}
%\vspace*{-1.0cm}
%\begin{multline}\label{ecsis3112}
%\int_{0}^{T/2}\left(\|y\|^2_{L^2(0,1)}+\|z\|^2_{L^2(0,1)}\right) 
\leq& \frac{T}{2}e^{CT}\int_{0}^{3T/4} \left(\|F_1\|^2_{L^2(0,1)}+\|F_2\|^2_{L^2(0,1)}\right)dt\\
&+e^{3CT/4}\int_{T/4}^{3T/4}\left(\|y(t)\|^2_{L^2(0,1)}+\|z(t)\|^2_{L^2(0,1)}\right)
\end{split}
\end{equation}

Now, integrating inequality \eqref{ecsiste3102} over $[0,t]$, where $t\in [0, T]$, we get
\begin{multline}\label{ecsiste3122}	\tilde m \int_{0}^{t}\left(\|\sqrt{a}y_x\|^2_{L^2(0,1)}+\|\sqrt{a}z_x\|^2_{L^2(0,1)}\right)\leq\frac{1}{2} \left(\|y\|^2_{L^2(0,1)}+\|z\|^2_{L^2(0,1)}\right)\\
+C\left[ \int_{0}^{t}\left(\|y\|^2_{L^2(0,1)}+\|z\|^2_{L^2(0,1)}\right)+ \int_{0}^{t}\left(\|F_1\|^2_{L^2(0,1)}+\|F_2\|^2_{L^2(0,1)}\right)\right].
\end{multline}

%Consider the double integral 
%$$I=\int_{T/2}^{3T/4}\int_{0}^{t}\left(\|\sqrt{a}y(\tau)\|^2_{L^2(0,1)}+\|\sqrt{a}z(\tau)\|^2_{L^2(0,1)}\right)d\tau dt$$
%and define
Since 
$J(t)=\int_{0}^{t}\left(\|\sqrt{a}y_x(\tau)\|^2_{L^2(0,1)}+\|\sqrt{a}z_x(\tau)\|^2_{L^2(0,1)}\right) d\tau$ is non-decreasing, for $t$ in the interval $[T/2,3T/4]$, we get
%Since $J(t)$ is a non-decreasing function, 
%(as increasing  integrates more non-negative terms), 
%for $t$ in the interval $[T/2,3T/4]$, we have that
%. For all $t\in[T/2,3T/4] $, we have 
%$J(t)\geq J(T/2)$.
%Similarly, for all $t\in[T/2,3T/4] $, we have  $J(t)\leq  J(3T/4)$.
%Using these properties, we can bound the integral  as follows:\\
%For , $t\in[T/2,3T/4] $, we have $J(t)\geq J(T/2)$. 
%Thus,
%$$I=\int_{T/2}^{3T/4}J(t)dt\geq \int_{T/2}^{3T/4}J(T/2)dt=\frac{T}{4}J(T/2)$$
%That is,
%$$I\geq \frac{T}{4}\int_{0}^{T/2}\left(\|\sqrt{a}y(\tau)\|^2_{L^2(0,1)}+\|\sqrt{a}z(\tau)\|^2_{L^2(0,1)}\right)d\tau$$
%Similarly, for all $t\in[T/2,3T/4] $, we have  $J(t)\leq  J(3T/4)$. Thus,
%$$I\leq \frac{T}{4}\int_{0}^{3T/4}\left(\|\sqrt{a}y(\tau)\|^2_{L^2(0,1)}+\|\sqrt{a}z(\tau)\|^2_{L^2(0,1)}\right)d\tau$$
%Combining both estimates, we obtain:
\begin{equation}\label{eq:compaJ}
\int_{0}^{T/2}\left(\|\sqrt{a}y_x(\tau)\|^2_{L^2(0,1)}+\|\sqrt{a}z_x(\tau)\|^2_{L^2(0,1)}\right)d\tau \leq J(t) \leq \int_{0}^{3T/4}\left(\|\sqrt{a}y_x(\tau)\|^2_{L^2(0,1)}+\|\sqrt{a}z_x(\tau)\|^2_{L^2(0,1)}\right)d\tau    
\end{equation}
%\hspace{-0.5cm}
%and
%$$
%\int_{0}^{T/2}\left(\|y(\tau)\|^2_{L^2(0,1)}+\|z(\tau)\|^2_{L^2(0,1)}\right)d\tau \leq G(t) \leq \int_{0}^{3T/4}\left(\|y(\tau)\|^2_{L^2(0,1)}+\|z(\tau)\|^2_{L^2(0,1)}\right)d\tau.
%$$

Thus, integrating inequality \eqref{ecsiste3122} from $T/2$ to $3T/4$ 
%and using \eqref{ecsis3112}
\begin{align*}
\tilde m \int_{T/2}^{3T/4} J(t)
%\int_{0}^{t}\left(\|\sqrt{a}y(\tau)\|^2_{L^2(0,1)}+\|\sqrt{a}z(\tau)\|^2_{L^2(0,1)}\right)d\tau dt 
&\leq \frac{1}{2}\int_{T/2}^{3T/4} \left(\|y\|^2_{L^2(0,1)}+\|z\|^2_{L^2(0,1)}\right)\\
&+C\left[\int_{T/2}^{3T/4} \int_{0}^{3T/4}\left(\|y\|^2_{L^2(0,1)}+\|z\|^2_{L^2(0,1)}\right)+ \int_{T/2}^{3T/4}\int_{0}^{3T/4}\left(\|F_1\|^2_{L^2(0,1)}+\|F_2\|^2_{L^2(0,1)}\right)\right] \\
&\leq \frac{1}{2}\int_{T/2}^{3T/4} \left(\|y\|^2_{L^2(0,1)}+\|z\|^2_{L^2(0,1)}\right)\\
&+\tilde C\left[\left(\int_{0}^{T/2}+ \int_{T/2}^{3T/4} \right)\left(\|y\|^2_{L^2(0,1)}+\|z\|^2_{L^2(0,1)}\right)+ \int_{0}^{3T/4}\left(\|F_1\|^2_{L^2(0,1)}+\|F_2\|^2_{L^2(0,1)}\right)\right]
\end{align*}
and using \eqref{eq:compaJ} and \eqref{ecsis3112} we get
\begin{multline}\label{ecsis3142}
\int_{0}^{T/2}\left(\|\sqrt{a}y_x\|^2_{L^2(0,1)}+\|\sqrt{a}z_x\|^2_{L^2(0,1)}\right)d\tau dt\leq\\
	C\left[\int_{T/4}^{3T/4} \left(\|y(t)\|^2_{L^2(0,1)}+\|z(t)\|^2_{L^2(0,1)}\right)dt+ \int_{0}^{3T/4}\left(\|F_1(t)\|^2_{L^2(0,1)}+\|F_2(t)\|^2_{L^2(0,1)}\right) dt\right].
\end{multline}

Finally, we observe that $e^{2sA}(s\lambda \zeta)^n$ and $e^{2sA}(s\lambda \sigma)^n$ are bounded in $[0, T/2]$ and $[T/4, 3T/4]$ respectively, for all $n\in \mathbb{Z}$. Hence, \eqref{ecsis3112}, \eqref{ecsis3142} and Carleman Inequality \eqref{desicarle} imply
\begin{eqnarray*}
\Gamma^1_0(y,z) &=&\int_{0}^{T/2}\int_{0}^{1} e^{2s A}((s\lambda) \zeta b^2 a (|y_x|^2 + |z_x|^2) + (s\lambda)^2 {\zeta}^2 b^2 (|y|^2 + |z|^2))dxdt\\
&\leq&C\int_{0}^{T/2}\int_{0}^{1} (  a (|y_x|^2 + |z_x|^2) +  (|y|^2 + |z|^2))dxdt\\
&\leq&C\left[\int_{T/4}^{3T/4}\int_{0}^{1} \left(|y|^2+|z|^2\right)dxdt+ \int_{0}^{3T/4}\int_{0}^{1}\left(|F_1|^2+|F_2|^2\right) dxdt\right]
\end{eqnarray*}
\vspace*{-0.6cm}
\begin{multline*}
	\leq C\left[\int_{T/4}^{3T/4}\int_{0}^{1}e^{2s \varphi} (s\lambda)^2 {\sigma}^2 b^2\left(|y|^2+|z|^2\right)dxdt\right.+
	 \left.\int_{0}^{3T/4}\int_{0}^{1}e^{2s A} (s\lambda)^4 {\zeta}^4 \left(|F_1|^2+|F_2|^2\right) dxdt\right]
\end{multline*}
and by Carleman inequality \eqref{carleman} we conclude the proof,
$$\Gamma^1_0(y,z) \leq C\left[\int_{0}^{T}\int_{0}^{1}e^{2s A}(s\lambda)^4 {\zeta}^4 \left(|F_1|^2+|F_2|^2\right) dxdt+
\int_{0}^{T}\int_{\omega}e^{2s \varphi} (s\lambda)^8 {\sigma}^8|y|^2  dxdt\right],$$
which implies the result.
\end{proof}

As a corollary we get the observability inequality:
\begin{coro} %\cite{DemarqueLimacoViana_deg_sys2020}
\label{cor:observability}
There exist positive constants $C$, $\lambda_0$ and $s_0$ such that, for any $s \geq s_0$, $\lambda \geq \lambda_0$ and any $\phi^T\in L^2(Q)$, the corresponding solution $(\phi,\psi)$ of \eqref{eq:adjoint_PDE} with $F_0 = F_1 = 0$, satisfies
\begin{equation}\label{Observability}
\|\phi(0)\|^2_{L^2(0,1)} + \|\psi(0)\|^2_{L^2(0,1)} \leq C \int_{\omega \times (0,T)} e^{2s A} s^8 \lambda^8 \zeta^8 |\phi|^2dxdt.
\end{equation}
\end{coro}
\begin{proof}
    See Corollary 1 of \cite{DemarqueLimacoViana_deg_sys2020}. 
\end{proof}
\color{black}

In the following sections we will need weights which depend only on $t$. Let us define
\begin{equation*}
\begin{cases}
A^*(t) = \displaystyle\max_{x \in (0,1)} A(x,t), \qquad \hat A(t) = \displaystyle\min_{x \in (0,1)} A(x,t), \\
\zeta^*(t) = \displaystyle\max_{x \in (0,1)} \zeta(x,t), \qquad \hat \zeta(t) = \displaystyle\min_{x \in (0,1)} \zeta(x,t),
\end{cases}
\end{equation*}
and we observe that $A^*(t) < 0$, $\hat A(t) < 0$ and that $\zeta^*(t) / \hat \zeta(t)$ does not depend on $t$ and is equal to some constant $\zeta_0 \in \R$. Moreover, if $\lambda$ is sufficiently large we can suppose
\begin{equation}
\label{eq:comp_pesos}
3 A^*(t) < 2 \hat A(t) < 0.
\end{equation}

%Let us define 
%$$
%\hat \Gamma(\phi,\varrho,\tilde \varrho) = \int_Q e^{2s \hat A}[(s\lambda) \hat \zeta b^2 a (|\phi_x|^2+|\varrho_x|^2+|\tilde \varrho_x|^2) + (s\lambda)^2 {\hat \zeta}^2 b^2 (|\phi|^2 + |\varrho|^2 + |\tilde \varrho|^2)]dxdt.
%$$
%Thus, 
Proposition \ref{prop:carleman} and Corollary \ref{cor:observability}  imply directly the following corollary where the weights depend only on $t$.

\begin{coro} %\cite{DemarqueLimacoViana_deg_sys2020}
\label{cor:carleman_pesos_t}
There exist positive constants $C$, $\lambda_0$ and $s_0$ such that, for any $s \geq s_0$, $\lambda \geq \lambda_0$ and any $\phi^T \in L^2(Q)$, the corresponding solution $(\phi,\psi)$ of \eqref{eq:adjoint_PDE} satisfies
\begin{equation}
\label{pesos_t_Carleman for eq:adjoint_optimality_system}
\|\phi(0)\|^2_{L^2(0,1)} + \hat \Gamma(\phi,\psi) \leq C \left( \int_Q e^{2s A^*} (\zeta^*)^4 (|F_1|^2+|F_2|^2)dxdt + \int_{\omega \times (0,T)} e^{2s A^*} (\zeta^*)^8 |\phi|^2 dxdt \right).
\end{equation}
\end{coro}

\section{Global null controllability of the linear system}
\label{sec:global_null_controllability}

Let us define the weights: 
\begin{equation}\label{eq:weights_rhos}
   \left\{ \begin{array}{l}
 \rho_0 = e^{-sA^*}  (\zeta^*)^{-2}, \qquad   \rho_1 = e^{-sA^*}  (\zeta^*)^{-4},\\  \rho_2 = e^{-3sA^*/2}  \hat \zeta^{-1},  \qquad \hat{\rho} = e^{-sA^*} (\zeta^*)^{-3},      
    \end{array}\right.
\end{equation}
satisfying
\begin{equation}\label{eq:compara_rhos}
 \rho_{1}\leq C \hat{\rho}\leq C\rho_{0}\leq C\rho_{2}, \qquad \hat{\rho}^{2}=\rho_{1}\rho_{0} \qquad \text{and} \qquad \rho_2 \leq C \rho_1^2.   
\end{equation}

In particular, Corollary \ref{cor:carleman_pesos_t} and \eqref{eq:comp_pesos} imply 
\begin{equation}
\label{eq:carleman_simples}
%\begin{split}
\|\phi(0)\|^2_{L^2(0,1)} +
\int_Q \rho_2^{-2} (|\phi|^2 + |\psi|^2)dxdt
%\hat \Gamma(\phi,\psi^1,\psi^2) 
\leq \ C \left( \int_Q \rho_0^{-2} (|F_1|^2+|F_2|^2)dxdt + \int_{\omega \times (0,T)} \rho_1^{-2} |\phi|^2 dxdt\right).
%\end{split}
\end{equation}

In the following theorem we show the global null controllability of the linear system \eqref{eq:PDE}. In particular, since the weight $\rho_0$ blows up at $t=T$, \eqref{estimate for solution} shows that $u(\cdot,T)=0$ in $[0,1]$. The same vanishing conclusion holds for $v$.

%Furthermore, estimates \eqref{des Proposition 5} and \eqref{des Proposition 6} give additional estimates verified by the derivatives of the states $(y,p^1,p^2)$ that are needed for the local null controllability of the nonlinear system in Section \ref{control for nonlinear system}.

\begin{teo}\label{theorem case linear}
If $u_{0}\in L^{2}(0,1)$, $\rho_2 H_1$, $\rho_2 H_2 \in L^2(Q)$, then there exists a control $h\in L^{2}(\omega \times (0,T))$ with associated states $u,v \in C^{0}([0,T];L^{2}(0,1))\cap L^{2}(0,T;H^{1}_{a})$, solutions of system \eqref{eq:PDE}, such that
\begin{equation}\label{estimate for solution}
\int_Q \rho_0^2 |u|^2 dxdt + \int_Q \rho_0^2 |v|^2 dxdt + \int_{\omega \times (0,T)} \rho_1^2 |h|^2dxdt \leq
C \kappa_{0}(H_{1},H_{2},u_{0}),   
\end{equation}
where $\kappa_{0}(H_{1},H_{2},u_{0})= |\rho_2 H_1|^2_{L^2(Q)} + |\rho_2 H_2|^2_{L^2(Q)} + |u_0|^2_{L_2(0,1)}$. In particular, $u(x,T)=0$ and $v(x,T)$, for all $x\in [0,1]$. 

%Furthermore, 
%\begin{equation}\label{des Proposition 5}
%    \begin{array}{c}
%\displaystyle\sup_{[0,T]}(\hat{\rho}^{2}\|y\|^{2}_{L^{2}(0,1)}) + \displaystyle\sup_{[0,T]}(\hat{\rho}^{2}\|p^{1}\|^{2}_{L^{2}(0,1)}) + \displaystyle\sup_{[0,T]}(\hat{\rho}^{2}\|p^{2}\|^{2}_{L^{2}(0,1)})\vspace{0.1cm}\\
%+\displaystyle\int_{Q}\hat{\rho}^{2} a(x)(|y_{x}|^{2} + |p^{1}_{x}|^{2}+|p^{2}_{x}|^{2})dxdt\ \leq C \kappa_{0}(H,H_{1},H_{2},y_{0})  
%    \end{array}
%\end{equation}
%and, if $y_{0}\in H^{1}_{a}(0,1)$ 
%\begin{equation}\label{des Proposition 6}
%    \begin{array}{c}
%\displaystyle\sup_{[0,T]}(\rho_{1}^{2}\|\sqrt{a}y_{x} \|^{2}_{L^{2}(0,1)})+ \displaystyle\sup_{[0,T]}(\rho_{1}^{2}\|\sqrt{a}p^{1}_{x} \|^{2}_{L^{2}(0,1)}) + \displaystyle\sup_{[0,T]}(\rho_{1}^{2}\|\sqrt{a}p^{2}_{x} \|^{2}_{L^{2}(0,1)})\\
%+ \displaystyle\int_{Q}\rho_{1}^{2}(|y_{t}|^{2}+|p^{1}_{t}|^{2}+|p^{2}_{t}|^{2}+|(a(x)y_{x})_{x}|^{2} + |(a(x)p^{1}_{x})_{x}|^{2}+ |(a(x)p^{2}_{x})_{x}|^{2})dxdt\\
%\leq C \kappa_{1}(H,H_{1},H_{2},y_{0}),  
%    \end{array}
%\end{equation}
%where $\kappa_{1}(H,H_{1},H_{1},y_{0})= |\rho_2 H|^2_{L^2(Q)} + |\rho_2 H_1|^2_{L^2(Q)} + |\rho_2 H_2|^2_{L^2(Q)} + \|y_0\|^2_{H^{1}_{a}}$. 
\end{teo}

%\textcolor{red}{Los coeficientes del sistema de optimalidad de orden uno y cero no son iguales para ambas ecuaciones, como en el caso de Miguel, para definir un único operador $L$ y $L^*$. Como salvar?}

\begin{proof}
Let us denote by
$$
L_1^*(\phi,\psi) =  -\phi_{t} - b(t)(a(x)\phi_{x})_{x}
- (d_1(x,t)\sqrt{a}\phi)_x + c_{11}(x,t)\phi + c_{21}(x,t)\psi, 
$$
$$
L_2^*(\phi,\psi) =  -\psi_{t} - b(t)(a(x)\psi_{x})_{x}
- (d_2(x,t)\sqrt{a}\psi)_x + c_{12}(x,t)\phi + c_{22}(x,t)\psi, 
$$
and
$$
\mathcal{L^*}(\phi,\phi) = (L_1^*(\phi,\psi),L_2^*(\phi,\psi)).
$$
Then, we define 
\begin{equation*}
\begin{array}{l}
    \mathcal{P}_{0} = \{ (\phi,\psi)\in C^{2}(\overline{Q})^{3}; \phi(0,t)=\phi(1,t) = 0\ \text{a.e in}\ (0,T),\ \psi(0,t)=\psi(1,t)=0\, \text{a.e in}\, (0,T),\\ \hspace{1.2cm}\psi(\cdot,0)=0\ \text{in}\ \Omega \}
\end{array}
\end{equation*}
and the application $b:\mathcal{P}_{0}\times \mathcal{P}_{0}\rightarrow \mathbb{R}$ given by
\begin{equation}
\label{eq:def_b}
%    \begin{array}{l}
%\hspace{-1.5cm}
b((\tilde{\phi},\tilde{\psi}),(\phi,\psi)) = \displaystyle\int_{Q}\rho_{0}^{-2} \mathcal{L}^*(\tilde \phi,\tilde \psi) \cdot \mathcal{L}^*(\phi,\psi) \ dx\ dt
      + \int_{\omega \times  (0,T)} \rho_{1}^{-2} \tilde{\phi}\phi\ dx\ dt,\quad \forall (\phi,\psi),(\tilde{\phi},\tilde{\psi})\in \mathcal{P}_{0},
%    \end{array}
\end{equation}
which is bilinear on $\mathcal{P}_{0}$ and defines an inner product. Indeed, taking $(\tilde{\phi},\tilde{\psi})=(\phi,\psi)$ in \eqref{eq:def_b}, we have that, by \eqref{eq:carleman_simples}  $b(\cdot,\cdot)$ is positive definite. The other properties are straightforwardly verified.

Let us consider the space $\mathcal{P}$ the completion of $\mathcal{P}_{0}$ for the norm associated to $b(\cdot,\cdot)$ (which we denote by $\|.\|_{\mathcal{P}}$). Then, $b(\cdot,\cdot)$ is symmetric, continuous and coercive bilinear form on $\mathcal{P}$.

Now, let us define the functional linear $\ell :\mathcal{P}\rightarrow\mathbb{R}$ as \begin{equation*}
    \langle\ell, (\phi,\psi)\rangle = \displaystyle\int_{0}^{1}y_{0}\phi(0)dx + \displaystyle\int_{Q}(H_1\phi + H_2\psi)dx dt.
\end{equation*}
Note that $\ell$ is a bounded linear form on $\mathcal{P}$. Indeed, applying the classical Cauchy-Schwartz inequality 
%$|u\cdot v|\leq |u||v|$, for $u,v\in 
in $\mathbb{R}^{4}$ and using \eqref{eq:carleman_simples}, we get
\begin{equation}\label{l limitado}
    \begin{array}{l}
       |\langle\ell, (\phi,\psi)\rangle|\leq  |y_{0}|_{L^{2}(0,1)}|\phi(0)|_{L^{2}(0,1)} + |\rho_{2}H_1|_{L^{2}(Q)}|\rho_{2}^{-1}\phi|_{L^{2}(Q)} + |\rho_{2}H_2|_{L^{2}(Q)}|\rho_{2}^{-1}\psi|_{L^{2}(Q)}\\
       \leq C \left(|y_{0}|^{2}_{L^{2}(0,1)} + |\rho_{2}H_{1}|^{2}_{L^{2}(Q)}  + |\rho_{2}H_{2}|^{2}_{L^{2}(Q)} \right)^{1/2}\left(b((\phi,\psi),(\phi,\psi))\right)^{1/2}\\
     \leq C(|y_{0}|^{2}_{L^{2}(0,1)} + |\rho_{2}H_{1}|^{2}_{L^{2}(Q)}  + |\rho_{2}H_{2}|^{2}_{L^{2}(Q)})^{1/2}\|(\phi,\psi)\|_{\mathcal{P}},
    \end{array}
\end{equation}
for all $(\phi,\psi) \in \mathcal{P}$. Consequently, in view of Lax-Milgram's theorem, there is only one $(\hat{\phi},\hat{\psi}) \in \mathcal{P}$ satisfying
\begin{equation}\label{eq: por Lax-M.}    b((\hat{\phi},\hat{\psi}),(\phi,\psi)) = \langle\ell, (\phi,\psi)\rangle,\quad \forall (\phi,\psi)\in\mathcal{P}.
\end{equation}

Let us set
\begin{equation}\label{definição de y, pi, h}
    \left\{\begin{array}{lll}
      u = \rho_{0}^{-2} L_1^{\ast} (\hat\phi,\hat\psi) & \text{in} & Q,\\
      v = \rho_{0}^{-2} L_2^{\ast} (\hat\phi, \hat{\psi}) & \text{in} & Q,\\
      h = -\rho_{1}^{-2}\hat{\phi}1_{O} &\text{in} & Q.
\end{array}\right.
\end{equation}

Then, replacing \eqref{definição de y, pi, h} in \eqref{eq: por Lax-M.} we have \\
%REVISAR:\\
\begin{equation*}
    \begin{array}{l}
         %\displaystyle\int_{Q}u B_1\,dx\,dt +
         \displaystyle\int_{Q}(u B_{1} + v B_{2})\,dx\,dt\\
         = \displaystyle\int_{0}^{1}y_{0}\phi(0)dx + \displaystyle\int_{\omega \times (0,T)} h \phi\,dx\,dt + \displaystyle\int_{Q}(H_1\phi + H_2\psi)dx\ dt,
    \end{array}
\end{equation*}
where $(\phi,\psi)$ is a solution of the system
\begin{equation*}
    \left\{\begin{array}{lll}
       L_1^{\ast} (\phi,\psi) = B_1   &\text{in}& Q,  \\
       L_2^{\ast} (\phi,\psi) = B_2   &\text{in}& Q, \\
       \phi(0,t)=\phi(1,t)=0 & \text{on} & (0,T), \\ \psi(0,t)=\psi(1,t)=0 & \text{on}& (0,T), \\
       \phi(\cdot,T)=0,\, \psi(\cdot,T)=0 &\text{in}& \Omega.
   \end{array}\right.
\end{equation*}

Therefore, $(u,v)$ is a solution by transposition of \eqref{eq:PDE}. Also, as $(\hat{\phi},\hat{\psi})\in\mathcal{P}$ and 
$H_{1}, H_{2}\in L^{2}(Q)$,  
using the well-posedness result of Appendix \ref{appendix A} applied to a linear equation we obtain
%Proposition \ref{Regularidade para linear system}, we obtain
$$u,v \in C^{0}([0,T];L^{2}(0,1))\cap L^{2}(0,T;H^{1}_{a}).$$
%\textcolor{red}{Esta faltando una proposition tipo la 3.1 del artículo con Joao y Suerlan. Esa proposicion tiene el label "Regularidade para linear system",por eso sale el ??. Corresponderia al Teorema \ref{teo34noauto} de este pdf.}

Moreover, from \eqref{l limitado} and \eqref{eq: por Lax-M.}
\begin{equation*}
\left(b((\hat{\phi},\hat{\psi}),(\hat{\phi},\hat{\psi}))\right)^{1/2}  \leq C \left(|y_{0}|^{2}_{L^{2}(0,1)} + |\rho_{2}H_{1}|^{2}_{L^{2}(Q)}  + |\rho_{2}H_{2}|^{2}_{L^{2}(Q)}\right)^{1/2}
\end{equation*}
that is,
\begin{equation*}
\begin{array}{l}
\displaystyle\int_{Q}\rho_{0}^{2} |u|^{2}\ dxdt + \displaystyle\displaystyle\int_{Q}\rho_{0}^{2} |v|^{2}\ dxdt + \displaystyle\int_{\omega \times (0,T)}\rho_{1}^{2} |h|^{2} dx dt\\
\leq C\left(|y_{0}|^{2}_{L^{2}(0,1)} + |\rho_{2}H_{1}|^{2}_{L^{2}(Q)}  + |\rho_{2}H_{2}|^{2}_{L^{2}(Q)} \right),
\end{array}
\end{equation*}
proving \eqref{estimate for solution}.
\end{proof}

\section{Final remarks and some problems}
\label{sec:final_remarks}

%Here we give some problems in the context addressed in this paper that, as far as we know, are open.
\begin{itemize}
	\item A natural application of our Carleman estimates are the case of controllability of equations defined on moving domains and free boundary. In fact, under regularity assumptions after a change of variables the original autonomous equation becomes non-autonomous. In the case of hierarchical control such as Stackelberg strategy with Nash equilibrium for the followers, a system of non-autonomous PDEs should be controlled and using a Carleman inequality for systems allows one to follow the computations as the autonomous case. This will be detailed in a subsequent work.
\item Consider the problem of a semilinear coupled degenerate non-autonomous parabolic equation with gradient terms. In one dimension, we consider the controllability of the following system:
%defined in the non-cilindrical domain $\widehat{Q}$ defined by a curve $\ell(t)$:
$$\begin{cases}
    u_{1t}-b(t)({a}(x)u_{1x})_{x}+F_1(u_1,u_2,u_{1x})= h\cara_{_{\mathcal{O}}}, & \ \ \ \text{in} \ \ \ Q,\\
	u_{2t}-b(t)({a}(x)u_{2x})_{x}+F_2(u_1,u_2,u_{2x})=0 & \ \ \ \text{in} \ \ \ Q,\\
	u_1(0,t)=u_1(1,t)=u_2(0,t)=u_2(1,t)=0, & \ \ \ \text{in} \ \ \ \Sigma,\\
	u_1(0)=u_1^0(x), \ u_2(0)=u_2^0(x) & \ \ \ \text{in} \ \ \ \Omega.
\end{cases}$$
After linearization, under conditions on the derivatives of $F_1$ and $F_2$, it is possible to apply our global controllability result. Then, one can apply an inverse function Theorem to get a local controllability result. In general this demands to obtain more estimates additional to \eqref{estimate for solution}.

%    \item An interesting problem is the case of a semilinear degenerate parabolic equation with gradient term in a non-cylindrical domain
%	$$	\begin{cases}
%		u_t-({a}(x')u_{x'})_{x'}+F(u,u_{x'})=\widehat{h}\cara_{{\widehat{\mathcal{O}}}}+\widehat{v}^1\cara_{{\widehat{\mathcal{O}}_1}}+\widehat{v}^2\cara_{{\widehat{\mathcal{O}}_2}}, & \ \ \ \text{in} \ \ \ \widehat{Q}\\
%		u(0,t)=u(\ell(t),t)=0, & \ \ \ \text{in} \ \ \ \widehat{\Sigma}\\
%		u(0)=u_0(x'), & \ \ \ \text{in} \ \ \ \Omega_0
%	\end{cases}$$
	\item Consider the null controllability of a quasilinear degenerate parabolic non-autonomous equation:
		$$	\begin{cases}
		u_t-b(t) (B(u){a}(x)u_{x})_{x}+F(u,u_{x})=h\cara_{{\widehat{\mathcal{O}}}}, & \ \ \ \text{in} \ \ \ Q\\
		u(0,t)=u(\ell(t),t)=0, & \ \ \ \text{in} \ \ \ \Sigma\\
		u(0)=u_0(x), & \ \ \ \text{in} \ \ \ \Omega
	\end{cases}$$
	Here $B = B(r)$ is a real function of class $C^3$ such that, for any $r\in \mathbb{R}$,
	$$0< B_0\leq B(r)\leq B_1, \ \ \ \ \ \ \ \ |B'(r)|+|B''(r)|+|B'''(r)|\leq M$$
%	\item We are interested in the controllability of the following degenerate, non-local, semi-linear parabolic equation in a non-cylindrical domain.
%	$$	\begin{cases}
%	u_t-\left({a}(x')l\left(\int\limits_{\Omega_t}u dx'\right)u_{x'}\right)_{x'}+F(u,u_{x'})=\widehat{h}\cara_{{\widehat{\mathcal{O}}}}+\widehat{v}^1\cara_{{\widehat{\mathcal{O}}_1}}+\widehat{v}^2\cara_{{\widehat{\mathcal{O}}_2}}, & \ \ \ \text{in} \ \ \ \widehat{Q}\\
%	u(0,t)=u(\ell(t),t)=0, & \ \ \ \text{in} \ \ \ \widehat{\Sigma}\\
%	u(0)=u_0(x'), & \ \ \ \text{in} \ \ \ \Omega_0
%\end{cases}$$

\item Another interesting problem is the null controllability of a fourth order  degenerate non-autonomous equation:
%in a non-cylindrical domain
$$
\begin{cases}
	u_t-b(t)\left({a}(x)u_{xx}\right)_{xx}=h\cara_{{\widehat{\mathcal{O}}}}, & \ \ \ \text{in} \ \ \ Q\\
	u(0,t)=u(\ell(t),t)=u_{xx}(0,t)=u_{xx}(\ell(t),t)=0, & \ \ \ \text{in} \ \ \ \Sigma\\
	u(0)=u_0(x'), & \ \ \ \text{in} \ \ \ \Omega
\end{cases}
$$

%\newpage
\appendix

\textbf{\section{Well-Posedness of system (\ref{eq:PDE}) }\label{appendix A}}

%Consider the optimality system (\ref{eqoptimal}) and using the change of variable $\varphi^{i}(t)=p^{i}(T-t)$, for $i=1,2.$, we obtain 
%\begin{equation*}
%	\begin{cases}
%		y_t-\frac{1}{\ell(t)^2}\left(a(x)y_x\right)_x-\frac{\ell'(t)}{\ell(t)}xy_x+F(y,\frac{1}{\ell(t)}y_x) = {h}\cara_{_{{\mathcal{O}}}}-\frac{1}{\mu_1}\varphi^1\cara_{_{\mathcal{O}_1}}-\frac{1}{\mu_2}\varphi^2\cara_{_{\mathcal{O}_2}}, & \ \ \ \text{in} \ \ \ {Q}\\
%		\varphi^i_t-\frac{1}{\ell(t)^2}\left(a(x)\varphi^i_x\right)_x+\frac{\ell'(t)}{\ell(t)}\left(x\varphi^i\right)_x + D_1F\left(y,\frac{1}{\ell(t)}y_x  \right)\varphi^i-\frac{1}{\ell(t)}\left(D_2F\left(y,\frac{1}{\ell(t)}y_x  \right)\varphi^i \right)_x  %F'(y)\varphi^i
%        =\alpha_i(y-y_{id})\cara_{_{{\mathcal{O}}_{id}}}, & \ \ \ \text{in} \ \ \ {Q}\\
%		y=0, \ \ \ \ \ \varphi^1=0, \ \ \ \ \ \varphi^2=0& \ \ \ \text{in} \ \ \ \sum\\
%		y(0)=y_0, \ \ \ \ \ \varphi^1(0)=\varphi_0^1, \ \ \ \ \ \ \varphi^2(0)=\varphi_0^2& \ \ \ \text{in} \ \ \ \Omega
%	\end{cases}
%\end{equation*}

In fact, we prove the well-posedness a system of semi-linear equations:
\begin{equation}\label{PDE_annexe}
	\left\{\begin{aligned}
		&u_t - b(t) \left(a(x) u_{x}\right)_{x} + d_1(x,t)\sqrt{a}u_x + F_1(u,v) = h 1_{\omega} + H_1 &&\text{in} && Q,\\
        &v_t - b(t) \left(a(x) v_{x}\right)_{x} + d_2(x,t)\sqrt{a}v_x + F_2(u,v) = H_2 &&\text{in}&& Q,\\
		&u(0,t)=u(1,t)=v(0,t)=v(1,t)=0&&\text{on} && (0,T), \\
		&u(\cdot,0) = u_0, \quad v(\cdot,0) = v_0 &&\text{in} && \Omega,
	\end{aligned}
	\right.
\end{equation}
%\begin{equation*}
%	\left\{\begin{aligned}
%		&y_t - b(t) \left(a(x) y_x\right)_x + c_1(x,t) \sqrt{a} y_x + F(y,c(t) \sqrt{a}y_x) = h 1_{O} - \sum_{i=1}^{2}\frac{1}{\mu_i} \varphi^i 1_{O_i} &&\text{in}&& Q, \\
%		&\varphi_t^i - b(t) \left( a(x) \varphi_x^i \right)_x + c_2(x,y,c(t)\sqrt{a} y_x,t) \sqrt{a} \varphi^i_x + g(x,y,c(t) \sqrt{a} y_x,t) \varphi^i   = \alpha_i (y-y_{i,d}) 1_{O_{i,d}}  &&\text{in}&& Q,\\
%		&y(0,t)=y(1,t)=0, \ \varphi^i(0,t) = \varphi^i(1,t) = 0 &&\text{on}&& (0,T), \\
%		& \varphi^i(\cdot,0) = \varphi^{i}_{0} &&\text{in }&& \Omega, \\
%		&y(\cdot,0) = y_0 &&\text{in}&& \Omega,
%	\end{aligned}
%	\right.
%\end{equation*}
where $d_1$, $d_2$ are bounded functions on $Q$, $0 < b_0 \leq b(t)$ bounded
%, $c(t)$ bounded   
%such that $|\frac{b_t}{b}|<C$, 
and $F_1$ and $F_2$ are globally Lipschitz in both variables.
%are $C^2$ and have bounded derivatives up to order 2. 
%In our case, $b(t) = \frac{1}{\ell(t)^2}$, $c_1 = -\frac{\ell(t)'}{\ell(t)} \frac{x}{\sqrt{a}}$, $c_2 = \frac{\ell(t)'}{\ell(t)} \frac{x}{\sqrt{a}} - c(t) D_2 F\left(y,c(t)\sqrt{a} y_x \right)$ and $g = \frac{\ell'(t)}{\ell(t)} + D_1 F \left(y,c(t) \sqrt{a} y_x \right) - c(t) \sqrt{a} \left( D_2 F\left(y, c(t) \sqrt{a} y_x \right) \right)_x$ and $\frac{x}{\sqrt{a}}$ is bounded.

Let $(w_{i})_{i}^{\infty}$ be an orthonormal basis of $H^{1}_{a}(0,1)$ such that 
\begin{equation*}
	-b(t)(a(x)w_{i,x})_{x}=\lambda_{i}w_{i}    
\end{equation*}

Fix $m\in\mathbb{N}^{*}$. Due the Caratheodory's theorem, there exist absolutely continuous functions $g_{im}=g_{im}(t)$ and $h_{im}=h_{im}(t)$ with $i\in\{1,2,...,m\}$ such that 
\begin{eqnarray*}
	t\in[0,T]\mapsto u_{m}(t)=\sum_{i=1}^{m}g_{im}(t)w_{i} \in H_{a}^{1}(0,1)
\end{eqnarray*}
and 
\begin{eqnarray*}
	t\in[0,T]\mapsto v_{m}(t)=\sum_{i=1}^{m}h_{im}(t)w_{i} \in H_{a}^{1}(0,1)
\end{eqnarray*}
satisfy
\begin{equation}
	\label{eq:galerkin_system}
	\left\{\begin{aligned}
		&(u_{m,t},w) - b(t)((a(x) u_{m,x})_x,w) + (d_1 \sqrt{a}  u_{m,x},w) + (F_1(u_m, v_{m}),w) = (h1_{\omega},w) + (H_1,w)  &&\text{in}&& Q, \\
		&(v_{m,t},\hat{w}) - b(t)((a(x) v_{m,x})_x,\hat{w}) + (d_2 \sqrt{a} v_{m,x},\hat w) + (F_2(u_m,v_m),\hat w)  = (H_2,\hat{w})   &&\text{in}&& Q,\\
		&u_{m}(0,t)=u_{m}(1,t)=0, \ v_{m}(0,t) = v_{m}(1,t) = 0 &&\text{on}&& (0,T), \\
		& u_{m}(\cdot,0)\to u_{0} &&\text{in }&& \Omega, \\
		& v_{m}(\cdot,0)\to v_{0} &&\text{in}&& \Omega,
	\end{aligned}
	\right.
\end{equation}
for any $w,\hat{w}\in [w_{1},w_{2},...,w_{m}]$  and $(\cdot,\cdot)=(\cdot,\cdot)_{L^{2}}$. Taking $w=u_{m}$ and $\hat{w}=v_{m}$, then 
\begin{equation*}
	\begin{split}
		&\frac{1}{2}\frac{d}{dt}\left(\|u_{m}\|^{2}+\|v_{m}\|^{2}\right) + b(t)\left(\|\sqrt{a}u_{m,x}\|^{2}+\|\sqrt{a}v_{m,x}\|^2\right) + (d_1 \sqrt{a} u_{m,x},u_m) + (d_2 \sqrt{a}v_{m,x},v_m)\\
        &+(F_1(u_m,v_m),u_m) + (F_2(u_m,v_m),v_m) =(h 1_\omega,u_{m})+(H_1,u_m) +(H_2,v_m).
	\end{split}
\end{equation*}

Applying the Cauchy-Schwarz inequality, using that $F_i$, $i=1,2$ are Lipschitz in both variables, that $c(t)$ is bounded and that $d_i$ are bounded functions on $Q$, we have, for some constant $C_* > 0$, 
\begin{equation*}
	\begin{split}
		&\frac{1}{2}\frac{d}{dt}\left(\|u_{m}\|^{2}+\|v_{m}\|^{2}\right) + b(t) \left(\|\sqrt{a}u_{m,x}\|^{2} + \|\sqrt{a}v_{m,x}\|^{2}\right)\\
		&\leq  C_{*}\left(\|h\|\|u_{m}\| + \|H_1\|\|u_{m}\|+\|H_2\|\|v_{m}\|  + \|u_m\|^2 + \|v_m\|^2 + \|\sqrt{a}u_{m,x}\|\|u_m\| + \|\sqrt{a}v_{m,x}\|\|v_m\| \right).
	\end{split}
\end{equation*}

Using the lower bound on $b$ and Young's inequality we obtain, for some $\epsilon>0$, 
\begin{equation*}
	%\label{systemwellposedness}
	\begin{split}
		&\frac{1}{2}\frac{d}{dt}\left(\|u_{m}\|^{2}+\|v_{m}\|^{2}\right) + b_0 \left(\|\sqrt{a}u_{m,x}\|^{2}+\|\sqrt{a}v_{m,x}\|^{2}\right)\\
		&\leq C_{*}\left(\|h\|^{2}+\|u_{m}\|^{2}+\|v_{m}\|^{2}+\sum_{i=1}^{2}\|H_i\|^{2} + \epsilon \|\sqrt{a} u_{m,x}\|^2 + \epsilon \|\sqrt{a} v_{m,x}\|^2 \right) 
	\end{split}
\end{equation*} 

Taking $\epsilon$ small and rearranging we get, for possibly another $C_*>0$, 
\begin{equation}\label{systemwellposedness}
	\begin{split}
		\frac{d}{dt}\left(\|u_{m}\|^{2}+\|v_{m}\|^{2}\right) + \left(\|\sqrt{a}u_{m,x}\|^{2}+\|\sqrt{a}v_{m,x}\|^{2}\right) \leq C_{*}\left(\|h\|^{2}+\|u_{m}\|^{2}+\|v_{m}\|^{2}+\sum_{i=1}^{2}\|H_i\|^{2} \right) 
	\end{split}
\end{equation} 

Integrating \eqref{systemwellposedness} from  to $0$ from $t$ and using Gronwall's inequality we deduce that, 
\begin{equation}
	\label{energyestimatewellpo1}
	\begin{split}
		&\|u_{m}(t)\|^{2}+\|v_{m}(t)\|^{2} + \int_{0}^{t}\|\sqrt{a}u_{m,x}\|^{2}ds + \int_{0}^{t}\|\sqrt{a}v_{m,x}\|^{2}ds\\ 
		&\leq e^{C_{*} T} \left(\|h\|^{2}_{L^{2}((0,T),L^2(\omega))}+\sum_{i=1}^{2}\|H_i\|^{2}_{L^{2}((0,T),L^2(0,1))} + \|u_{0}\|_{H^{1}_{a}(0,1)} +\|v_0\|^{2}_{H^{1}_a(0,1)}\right)
	\end{split}
\end{equation}

Since \eqref{energyestimatewellpo1} holds for any $t\in[0,T]$, 
\begin{equation}\label{energyestimate1}
	\begin{split}
		&\|u_m\|_{L^{\infty}(0,T,L^{2}(0,1))}^{2}+\|v_{m}\|_{L^{\infty}(0,T,L^{2}(0,1))}^{2}+\|\sqrt{a}u_{m,x}\|_{L^{2}(0,T,L^{2}(0,1))}^{2}+\|\sqrt{a}v_{m,x}\|_{L^{2}(0,T,L^{2}(0,1))}^{2}\\ 
		&\leq e^{C_{*} T} \left(\|h\|^{2}_{L^{2}((0,T),L^2(\omega))}+\sum_{i=1}^{2}\|H_i\|^{2}_{L^{2}((0,T),L^2(0,1))}+\|u_{0}\|_{H^{1}_{a}(0,1)}+\|v_{0}\|^{2}_{H^{1}_a(0,1)}\right) =: \mathcal{K}_1.
	\end{split}
\end{equation}

Estimate II: Taking $w=u_{m,t}$ and $\hat{w}=v_{m,t}$ in \eqref{eq:galerkin_system}  we get
\begin{equation*}
	\begin{split}
		&\|u_{m,t}\|^2 + \|v_{m,t}\|^2 + b(t) \frac{1}{2}\frac{d}{dt} \left( \|\sqrt{a}u_{m,x}\|^2 + \|\sqrt{a}v_{m,x}\|^2 \right) + (d_1 \sqrt{a} u_{m,x},u_{m,t}) + (d_2 \sqrt{a}v_{m,x},v_{m,t})\\
        &+(F_1(u_m,v_{m}),u_{m,t}) + (F_2(u_m,v_m),v_{m,t}) \leq \|h\|^2 + \sum_{i=1}^2  \|H_i\|^2 + \frac{1}{2}\|u_{m,t}\|^2 + \frac{1}{4} \|v_{m,t}\|^2.
	\end{split}
\end{equation*}

Thus, the lower-boundedness if $b(t)$, Cauchy-Schwarz inequality, the fact that $F_i$, $i=1,2$, are Lipschitz in both variables, that $d_i$, $i=1,2$, are bounded functions on $Q$ and Young's inequality, imply that there exists a constants $D>0$ and $\epsilon>0$ such that 
\begin{equation*}
	\begin{split}
		&\|u_{m,t}\|^2 + \|v_{m,t}\|^2 + \frac{d}{dt} \left( \|\sqrt{a}u_{m,x}\|^2 + \|\sqrt{a}v_{m,x}\|^2 \right) \\
		&\leq D \left[ \|h\|^2 + \sum_{i=1}^2  \|H_i\|^2 + \|\sqrt{a}u_{m,x}\|^2 + \|\sqrt{a}v_{m,x}\|^2 + \|u_m\|^2 + \|v_m\|^2 \right].  
	\end{split}
\end{equation*}

%Thus, rearranging,
%\begin{equation*}
%	\begin{split}
%		&\|y_{m,t}\|^2 + \sum_{i=1}^2  \|\varphi^i_{m,t}\|^2 + \frac{d}{dt} \left( \|\sqrt{a}y_{m,x}\|^2 + \sum_{i=1}^2  \|\sqrt{a}\varphi^i_{m,x}\|^2 \right) \\
%		&\leq D \left[ \|h\|^2 + \sum_{i=1}^2  \|y_{i,d}\|^2 + \|y_m\|^2 + \sum_{i=1}^2  \|\varphi^i_m\|^2 + \|\sqrt{a}y_{m,x}\|^2 + \sum_{i=1}^2 \|\sqrt{a}\varphi^i_{m,x}\|^2 \right].  
%	\end{split}
%\end{equation*}

Integrating in $t$ on $[0,t]$  and applying Gronwall's inequality 
we get
\begin{equation*}
	\begin{split}
		&\|u_{m,t}\|^2_{L^2(0,T,L^2(0,1))} +  \|v_{m,t}\|^2_{L^2(0,T,L^2(0,1))} + \left( \|\sqrt{a}y_{m,x}(t)\|^2 + \|\sqrt{a}v_{m,x}(t)\|^2 \right) \\
		&\leq e^{DT} \left[ \|\sqrt{a}u_{m,x}(0)\|^2 + \|\sqrt{a}v_{m,x}(0)\|^2 +\|h\|^2_{L^2(0,T,L^2(\omega))} + \sum_{i=1}^2  \|H_i\|^2_{L^2(0,T,L^2(0,1))} + \|u_m\|^2_{L^\infty(0,T,L^2(0,1))}T   \right. \\
		&\left. +  \|v_m\|^2_{L^\infty(0,T,L^2(0,1))}T   \right].
		%\\
		%&\leq e^{DT} \left[ \|\sqrt{a}y_{m,x}(0)\|^2 + \sum_{i=1}^2  \|\sqrt{a}\varphi^i_{m,x}(0)\|^2 +\|h\|^2_{L^2(0,T,L^2(\omega))} + \sum_{i=1}^2  \|y_{i,d}\|^2_{L^2(0,T,L^2(O_i))} + \|y_m\|^2_{L^\infty(0,T,L^2(0,1))}T \right. \\
		%&\left. + \sum_{i=1}^2  \|\varphi^i_m\|^2_{L^\infty(0,T,L^2(0,1))}T \right] e^{D \int_0^t \left(  \|\sqrt{a}y_{m,x}\|^2 + \sum_{i=1}^2  \|\sqrt{a}\varphi^i_{m,x}\|^2 \right) ds}
	\end{split}
\end{equation*}

Finally, using estimate I,
\begin{equation*}
	\begin{split}
		&\|u_{m,t}\|^2_{L^2(0,T,L^2(0,1))} +  \|v_{m,t}\|^2_{L^2(0,T,L^2(0,1))} + \|\sqrt{a}u_{m,x}\|^2_{L^\infty(0,T,L^2(0,1))} + \|\sqrt{a}v_{m,x}\|^2_{L^\infty(0,T,L^2(0,1))} \\
		&\leq e^{DT} \left[ \|\sqrt{a}u_{m,x}(0)\|^2_{L^2(0,1)} + \|\sqrt{a}v_{m,x}(0)\|^2_{L^2(0,1)} +\|h\|^2_{L^2(0,T,L^2(\omega))}  
		%\right. \\
		%&\left. 
		+ \sum_{i=1}^2  \|H_i\|^2_{L^2(0,T,L^2(0,1))} + 2\mathcal{K}_1 T  \right] =:\mathcal{K}_2.
	\end{split}
\end{equation*}

Estimate III: Taking $w=-(a(x)u_{m,x})_x$ and $\hat{w}=-(a(x) v_{m,x})_x$ in \eqref{eq:galerkin_system}, and proceeding as in the above estimates, we get
\begin{equation*}
	\begin{split}
		&-(u_{m,t}, (a u_{m,x})_x) - (v_{m,t}, (a v_{m,x})_x) + b(t) \left(\| (a u_{m,x})_x \|^2 + \| (a v_{m,x})_x \|^2 \right) \\
		&- (d_1 \sqrt{a} u_{m,x},(a(x)u_{m,x})_x) - (d_2 \sqrt{a}v_{m,x},(a(x) v_{m,x})_x) \\
		& - \sum_{i=1}^2 (F_i(u_m, v_{m}),(a(x)u_{m,x})_x)   \leq C_\epsilon \|h\|^2 + C_\epsilon\sum_{i=1}^2  \|H_i\|^2 + \epsilon\|(a u_{m,x})_x\|^2  + \epsilon \|(a v_{m,x})_x\|^2.
	\end{split}
\end{equation*}

Thus, rearranging, using the lower-boundedness if $b(t)$, Cauchy-Schwarz inequality, the fact that $F_i$, $i=1,2$, are Lipschitz in both variables, that $d_i$ are bounded functions on $Q$ and Young's inequality, we get that there exists a constants $D>0$ such that 
\begin{equation*}
	\begin{split}
		&\frac{d}{dt}\left( \|\sqrt{a} u_{m,x} \|^2 + \|\sqrt{a}v_{m,x}\|^2\right) + \|(a u_{m,x})_x\|^2 + \|(a v_{m,x})_x\|^2 \\
		&\leq D \left( \|h\|^2 + \sum_{i=1}^2  \|H_{i}\|^2 + \|u_m\|^2 + \|v_m\|^2
		+ \|\sqrt{a} u_{m,x} \|^2 + \|\sqrt{a} v_{m,x}\|^2 
		%+ \sum_{i=1}^2 \|\sqrt{a}y_{m,x}\|^2\|\sqrt{a}\varphi^i_{m,x}\|^2
		\right).
	\end{split}
\end{equation*}

Integrating in $t$ on $[0,t]$, using Gronwall's inequality and proceeding as in estimate II, we have that
\begin{equation*}
	\begin{split}
		\|\sqrt{a} u_{m,x}\|^2_{L^\infty(0,T,L^2(0,1))}  + \|\sqrt{a}v_{m,x}\|^2_{L^\infty(0,T,L^2(0,1))} +  \|(a u_{m,x})_x\|^2_{L^2(0,T,L^2(0,1))} + \|(a v_{m,x})_x\|^2_{L^2(0,T,L^2(0,1))} \\
		\leq e^{DT} \left[ \|\sqrt{a} u_{m,x}(0) \|^2 + \|\sqrt{a} v_{m,x}(0)\|^2 + \|h\|^2_{L^2(0,t,L^2(\omega))} + \sum_{i=1}^2  \|H_{i}\|^2_{L^2(0,t,L^2(0,1)} + 4\mathcal{K}_1 \right] =: \mathcal{K}_3.
	\end{split}
\end{equation*}

Since $\mathcal{K}_1$, $\mathcal{K}_2$ and $\mathcal{K}_3$ do not depend on $m$, the three estimates above imply that the sequences $(u_m)$ and $(v_m)$ are bounded in 
$$
L^2(0,T,L^2(0,1)) \cap H^1(0,T,H^2_a(0,1)).
$$

Therefore, there exist subsequences $(u_{m_j})$, $(v_{m_j})$ such that
$$
u_{m_j} \rightharpoonup u,\qquad v_{m_j} \rightharpoonup v, \qquad \text{as } j\to \infty, 
$$
weakly in $L^2(0,T,L^2(0,1)) \cap H^1(0,T,H^2_a(0,1))$.
In fact, since the immersion $H^2_a(0,1)$ in $H^1_a(0,1)$ is compact, by the theorem of Aubin-Lions we get that
$$
u_{m_j} \to u,\qquad v_{m_j} \to v, \qquad \text{as } j\to \infty, 
$$
strongly in $H^1_a(0,1)$. Using the continuity of $F_i$,
at least for a subsequence, we can pass to the limit in $m$ in  all the terms of the the approximate system \eqref{eq:galerkin_system}.

The uniqueness is proved by standard methods for nonlinear systems, taking into consideration the bounds established in estimates I to III for $u$, $u_x$ and $v$ and $v_x$, for $i=1,2$.\\

%\section{Carleman's inequality}\label{appendix B}

\end{itemize}
\bibliographystyle{abbrv}
\bibliography{referencias}

\begin{thebibliography}{10}

\bibitem{AitBenHassi-11}
E.~Ait Ben~Hassi, F.~Ammar~Khodja, A.~Hajjaj, and L.~Maniar.
\newblock Null controllability of degenerate parabolic cascade systems.
\newblock {\em Port. Math.}, 68(3):345--367, 2011.

\bibitem{AitBenHassi-13}
E.~Ait Ben~Hassi, F.~Ammar~Khodja, A.~Hajjaj, and L.~Maniar.
\newblock Carleman estimates and null controllability of coupled degenerate
  systems.
\newblock {\em Evol. Equ. Control Theory}, 3:441--459, 2013.

\bibitem{fragnelli_robin_BC-25}
M.~Akil, G.~Fragnelli, and S.~Ismail.
\newblock Non-autonomous degenerate parabolic equations with robin boundary
  conditions: Carleman estimates and null-controllability.
\newblock {\em Appl Math Optim}, 91(35.2), 2025.

\bibitem{fragnelli_no_autonomo}
M.~Akil, G.~Fragnelli, and S.~Ismail.
\newblock Null-controllability and carleman estimates for non-autonomous
  degenerate pdes: a climatological application.
\newblock {\em Journal of Mathematical Analysis and applications}, 543(2, Part
  1):128984, 2025.

\bibitem{Alabau_cannarsa_fragnelli-06}
F.~Alabau-Boussouira, P.~Cannarsa, and G.~Fragnelli.
\newblock Carleman estimates for degenerate parabolic operators with
  applications to null controllability.
\newblock {\em Journal of Evolution Equations}, 6:161 -- 204, 2006.

\bibitem{Cannarsa_Martinez_Vancostenoble-2005}
P.~Cannarsa, P.~Martinez, and V.~J.
\newblock Null controllability of the degenerate heat equations.
\newblock {\em Adv. Differential Equations}, 10:153--190, 2005.

\bibitem{DemarqueLimacoViana_deg_eq2018}
R.~Demarque, J.~Límaco, and L.~Viana.
\newblock Local null controllability for degenerate parabolic equations with
  nonlocal term.
\newblock {\em Nonlinear Analysis: Real World Applications}, 43:523 -- 547,
  2018.

\bibitem{DemarqueLimacoViana_deg_sys2020}
R.~Demarque, J.~Límaco, and L.~Viana.
\newblock Local null controllability of coupled degenerate systems with
  nonlocal terms and one control force.
\newblock {\em Evolution Equations and Control Theory}, 9:605–635, 2020.

\bibitem{Feller-1952}
W.~Feller.
\newblock The parabolic differential equations and the associated semi-groups
  of transformations.
\newblock {\em Annals of Mathematics}, 55:468--519, 1952.

\bibitem{Feller-1954}
W.~Feller.
\newblock Diffusion processes in one dimension.
\newblock {\em Transactions of the American Mathematical Society}, 55:1--31,
  1954.

\bibitem{GlobalCarleman-06}
E.~Fernandez-Cara and S.~Guerrero.
\newblock Global carleman inequalities for parabolic systems and applications
  to controllability.
\newblock {\em SIAM J. Control Optim.}, 45.4:1395 -- 1446, 2006.

\bibitem{FLORIDIA-2014}
G.~Floridia.
\newblock Approximate controllability for nonlinear degenerate parabolic
  problems with bilinear control.
\newblock {\em Journal of Differential Equations}, 257(9):3382--3422, 2014.

\bibitem{Fur_Ima-96}
I.~O. Fursikov A.~V.
\newblock {\em Controllability of evolution equations, Lecture Note Series 34,
  Research Institute of Mathematics}, volume~34.
\newblock Seoul National University, 1996.

\bibitem{Martinez_Raymond_Vancostenoble-2003}
P.~Martinez, J.~Raymond, and J.~Vancostenoble.
\newblock Regional null controllability for a linearized crocco type equation.
\newblock {\em SIAM J. Control Optim}, 42:709--728, 2003.

\end{thebibliography}


\begin{thebibliography}{99}

\setlength{\itemsep}{0cm}

\begin{small}

\bibitem{Fur_Ima-96} Fursikov, A. V., Imanuvilov, O.Y.: Controllability of evolution equations, Lecture Note Series 34, Research Institute of Mathematics. Seoul National University, Seoul (1996).

%\bibitem{LSU-68} Ladyzhenskaya, O. A., Solonnikov V. A., Uralceva, N. N.: Linear and quasilinear equations of parabolic type, Translations of Mathematical Monographs, vol. 23, American Mathematical Society, Providence, R.I. (1968).

%\bibitem{Liu_Zhang-12} Liu, X., and Zhang, X.: Local controllability of multidimensional quasi-linear parabolic equations, SIAM J. Control Optim. 50 (4), 2046-2064 (2012).

\bibitem{Alabau_cannarsa_fragnelli-06} F. Alabau-Boussouira, P. Cannarsa, and G. Fragnelli. Carleman estimates for degenerate parabolic operators with applications to null controllability. Journal of Evolution Equations, 6(2):161–204, 2006.

\bibitem{Alekseev} V. M. Alekseev, V. M. Tikhomorov, and S. V. Formin. Optimal control. Contemporary Soviet Mathematics, Consultants Bureau, New York, 1987.

\bibitem{Joao_Juan_Suerlan-25} J.C. Barreira, J.B. Límaco, S. Silva, L.P. Yapu, Hierarchical null controllability of a degenerate parabolic
equation with non-local coefficient, 2025. (Submitted)

\bibitem{DemarqueLimacoViana_deg_eq2018} R. Demarque, J. Límaco, and L. Viana. Local null controllability for degenerate parabolic equations with nonlocal term. Nonlinear Analysis: Real World Applications, 43:523–547,
2018.

\bibitem{DemarqueLimacoViana_deg_sys2020} R. Demarque, J. Límaco, and L. Viana. Local null controllability of coupled degenerate systems with nonlocal terms and one control force. Evolution Equations and Control Theory, 9(3):605–635, 2020.

\bibitem{djomegne2022hierarchical} L. Djomegne, C. Kenne, R. Dorville, and P. Zongo. Hierarchical null controllability of a semi-linear degenerate parabolic equation with a gradient term. arXiv preprint arXiv:2209.12450, 2022.

\bibitem{FadiliManiar} M. Fadili and L. Maniar. Null controllability of n-coupled degenerate parabolic systems with
m-controls. J. Evol. Equ., 17:1311–1340, 2017.


\end{small}

\end{thebibliography}

\begin{comment}

\end{comment}

\end{document}